\documentclass[reqno,11pt]{amsart}

\usepackage{latexsym}
\usepackage[margin=3cm]{geometry}
\usepackage{amscd}
\usepackage{amsthm}
\usepackage{marginnote}
\usepackage{subfigure,amsmath,amsfonts,amssymb,epsfig}
\usepackage{graphics}
\usepackage{xy}
\usepackage{mathrsfs}
\usepackage{tipa}

\usepackage[latin1]{inputenc}

\usepackage{subfigure,amsmath}

\input xy
\xyoption{all}

\newtheorem{theorem}{Theorem}[section]

\newtheorem{proposition}[theorem]{Proposition}

\theoremstyle{remark}
\newtheorem{remark}[theorem]{Remark}

\numberwithin{equation}{section}
\allowdisplaybreaks

\author[Michael J.\ Schlosser]{Michael J.\ Schlosser$^*$}
\address{Fakult\"at f\"ur Mathematik, Universit\"at Wien,
Oskar-Morgenstern-Platz~1, A-1090 Vienna, Austria}
\email{michael.schlosser@univie.ac.at}
\urladdr{http://www.mat.univie.ac.at/{\textasciitilde}schlosse}
\thanks{$^*$Partly supported by FWF Austrian Science Fund
grant F50-08 within the SFB
``Algorithmic and enumerative combinatorics''.} 

\author[Meesue Yoo]{Meesue Yoo$^{**}$}
\address{Fakult\"at f\"ur Mathematik, Universit\"at Wien,
Oskar-Morgenstern-Platz~1, A-1090 Vienna, Austria}
\email{meesue.yoo@univie.ac.at}
\thanks{$^{**}$Fully supported by FWF Austrian Science Fund
grant F50-08 within the SFB
``Algorithmic and enumerative combinatorics''.}

\title[Elliptic product formula for augmented rook boards]{An elliptic
extension of the general product formula for augmented rook boards}



\newcommand{\A}{\mathcal A}
\newcommand{\B}{\mathcal B}
\newcommand{\N}{\mathcal N}
\newcommand{\F}{\mathcal F}

\newcommand{\I}{\text{\sl\bf\textroundcap{\i}}}
\newcommand{\J}{\text{\sl\bf\textroundcap{\j}}}

\def\!{\mskip-\thinmuskip}
\def\,{\mskip\thinmuskip}
\def\;{\mskip\thickmuskip}

\begin{document}

%%%%%%%%%%%%%%%%%%%%%%%%%%5
 \begin{abstract}
Rook theory has been investigated by many people since its
introduction by Kaplansky and Riordan in 1946. Goldman, Joichi and 
White in 1975 showed that the sum over $k$ of the product of
the $(n-k)$-th rook numbers multiplied by the $k$-th falling factorial
polynomials factorize into a product.
In the sequel, different types of generalizations and analogues
of this product formula have been derived by various authors.
In 2008, Miceli and Remmel
constructed a rook theory model involving augmented rook boards
in which they showed the validity of a general product formula
which can be specialized to all other product formulas
that so far have appeared in the literature on rook theory. 
In this work, we construct an elliptic extension of the $q$-analogue 
of Miceli and Remmel's result. Special cases yield 
elliptic extensions of various known rook theory models.
 \end{abstract}
%%%%%%%%%%%%%%%%%%%%%%%%%%%%%5
 
 \maketitle

 %%%%%%%%%%%%%%%%%%%%%%%%%%%%%%%%%%%%%%%%%55
\section{Introduction}

%%%%%%%%%%%%%%%%%%%%%%%%%%%%%%%%%%%%%%%%%%%%%%5
% introduction to rook theory 
% elliptic functions
% definition of weight functions 

Let $\mathbb{N}$ be the set of positive integers and $[n]=\{1,2,\dots, n\}$.
A \emph{board} is a finite subset of the $\mathbb{N}\times \mathbb{N}$ grid
of squares. We label the rows with $1,2,\dots$ from bottom to top, and the 
columns from left to right. To denote the cell in the column $i$ and row $j$
we use the notation $(i,j)$.  
For a sequence of nonnegative integers $b_i\ge 0$, $i=1,\dots, n$,
let $B(b_1,\dots, b_n)$ 
denote the following set of cells 
$$B(b_1,\dots, b_n)=\{(i,j) : 1\le i\le n, ~1\le j \le b_i\}.$$
If a board $B$ can be represented by $B(b_1,\dots, b_n)$ for some $b_i$'s, 
then $B$ is called a \emph{skyline board}. Especially when $b_i$'s are  
nondecreasing, that is, $0\le b_1\le b_2\le \cdots \le b_n$,
then the board is called 
a \emph{Ferrers board}. 
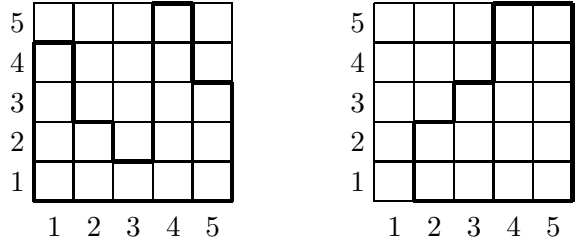
\begin{figure}[ht]

$$\begin{picture}(85,85)(0,-10)
\multiput(10,0)(0,15){6}{\line(1,0){75}}
\multiput(10,0)(15,0){6}{\line(0,1){75}}
\put(1,64){5}
\put(1,49){4}
\put(1,34){3}
\put(1,19){2}
\put(1,4){1}
\put(15, -13){1}
\put(30,-13){2}
\put(45,-13){3}
\put(60,-13){4}
\put(75,-13){5}
%\put(12,-23){$B(4,2,1,5,3)$}
\thicklines \linethickness{1.3pt}
\multiput(10,0)(0,15){1}{\line(0,1){60}}
\multiput(10,60)(0,15){1}{\line(1,0){15}}
\multiput(25,30)(0,15){1}{\line(0,1){30}}
\multiput(25,30)(0,15){1}{\line(1,0){15}}
\multiput(40,15)(0,15){1}{\line(0,1){15}}
\multiput(40,15)(0,15){1}{\line(1,0){15}}
\multiput(55,15)(0,15){1}{\line(0,1){60}}
\multiput(55,75)(0,15){1}{\line(1,0){15}}
\multiput(70,45)(0,15){1}{\line(0,1){30}}
\multiput(70,45)(0,15){1}{\line(1,0){15}}
\multiput(85,0)(0,15){1}{\line(0,1){45}}
\multiput(10,0)(0,15){1}{\line(1,0){75}}
\end{picture}
\qquad \qquad
\begin{picture}(85,85)(0,-10)
\multiput(10,0)(0,15){6}{\line(1,0){75}}
\multiput(10,0)(15,0){6}{\line(0,1){75}}
\put(1,64){5}
\put(1,49){4}
\put(1,34){3}
\put(1,19){2}
\put(1,4){1}
\put(15, -13){1}
\put(30,-13){2}
\put(45,-13){3}
\put(60,-13){4}
\put(75,-13){5}
%\put(14,-23){$B(0,2,3,5,5)$}
\thicklines \linethickness{1.3pt}
\multiput(25,0)(0,15){1}{\line(0,1){30}}
\multiput(25,30)(0,15){1}{\line(1,0){15}}
\multiput(40,30)(0,15){1}{\line(0,1){15}}
\multiput(40,45)(0,15){1}{\line(1,0){15}}
\multiput(55,45)(0,15){1}{\line(0,1){30}}
\multiput(55,75)(0,15){1}{\line(1,0){30}}
\multiput(25,0)(0,15){1}{\line(1,0){60}}
\multiput(85,0)(0,15){1}{\line(0,1){75}}
\end{picture}$$
\caption{A skyline board $B(4,2,1,5,3)$ and a Ferrers board $B(0,2,3,5,5)$.}
\end{figure}

Given a board $B$, we say that we place a rook in the $(i,j)$ cell 
for choosing the cell $(i,j)\in B$. A rook attacks the cells in the same row 
and in the same column. Thus a \emph{nonattacking} rook placement of $k$
rooks in $B$
is the subset of $k$ cells in $B$ such that no two cells have the same 
row coordinate or the same column coordinate. Let $\N _k(B)$ denote the set 
of all nonattacking rook placements of $k$ rooks in $B$, and
$r_k (B)=|\N _k(B)|$. 
Then for a Ferrers board $B=B(b_1,\dots, b_n)$, Goldman, Joichi and
White~\cite{GJW} proved that 
\begin{equation}\label{eqn:GJW}
\prod_{i=1}^n (z+b_i -i+1)=\sum_{k=0}^n r_{n-k}(B) (z)\!\downarrow_k, 
\end{equation}
where $(z)\!\downarrow_k=z(z-1)\cdots (z-k+1)$. 

Garsia and Remmel \cite{GR} developed a $q$-analogue of the rook theory 
by introducing a rook cancellation scheme. For the $q$-rook theory,
we assume that the  given board is a Ferrers board.
Given a Ferrers board $B$, a rook in $B$ cancels all the cells in
the same row to the right and the cells in the same column below it. 
For a rook placement $P\in \N _k(B)$, let $u_B (P)$ be the number of
uncancelled cells in $B-P$. The $q$-analogue of the $k$-th rook number
defined by Garsia and Remmel is 
$$r_k(q;B) =\sum_{P\in \N _k(B)}q^{u_B (P)}.$$
Then we have 
\begin{equation}\label{eqn:qrook}
 \prod_{i=1}^n [z+b_i -i+1]_q=\sum_{k=0}^n r_{n-k}(B) [z]_q\!\downarrow_k, 
\end{equation}
where $[z]_q=\frac{1-q^z}{1-q}$ and
$[z]_q\!\downarrow_k=[z]_q[z-1]_q\cdots [z-k+1]_q$. 

Garsia and Remmel \cite{GR1} also defined \emph{file numbers}.
Given a board $B$, a \emph{file placement} of $k$ rooks is a $k$-subset
of $B$ such that no two cells lie in the same column, that is, we can
choose two cells in the same row, but each column contains at most one rook.
Let $\F _k(B)$ denote the set of all $k$-file placements in $B$ and
$f_k (B)=|\F _k (B)|$. Then for any skyline board $B=B(b_1,\dots, b_n)$,
we have the product formula
\begin{equation}\label{eqn:file}
\prod_{i=1}^n (z+b_i) =\sum_{k=0}^n f_{n-k}(B) z^k, 
\end{equation}
and the $q$-analogue 
\begin{equation}\label{eqn:qfile}
\prod_{i=1}^n [z+b_i]_q =\sum_{k=0}^n f_{n-k}(q;B) [z]_q ^k, 
\end{equation}
where $f_k (q;B)$ is defined by 
$$f_k (q;B)=\sum_{P\in \F_k (B)}q^{\tilde{u}_B (P)},$$
and $\tilde{u}_B (P)$ is the number of cells lying above a rook in $P$ or cells
in columns containing no rooks. 

There have been many generalizations of rook numbers and respective product
formulas.  Goldman and Haglund~\cite{GH} introduced the $i$-creation model in
which a rook creates  $i$ new rows to the right and proved a product formula
involving the $i$-rook number. The $\alpha$-parameter 
model that Goldman and Haglund defined in \cite{GH} as well gives the
$q$-analogue of the $i$-creation model. Haglund and Remmel~\cite{HR}
considered shifted rook boards and defined the 
rook placement as a subset of some perfect matching in the complete graph. 
Remmel and Wachs \cite{RW} defined the $\J$-attacking rook model and 
proved a product formula involving factors of $\J$-differences. 
Briggs and Remmel \cite{BR0} considered the rook model corresponding to
partial permutations of the wreath product of the cyclic group of order
$m$ with the symmetric group $S_n$, $C_m\wr S_n$. 
Briggs and Remmel~\cite{BR1} also considered one more parameter $p$ and
defined a $p,q$-analogue of the rook numbers and proved a
respective product formula. 

Each of the above models can be obtained by specializing the rook model
of Miceli and Remmel~\cite{MR}. The main purpose of this work is to
construct an elliptic extension of the rook model of Miceli and Remmel
which can be specialized to give elliptic extensions of all the known
rook models mentioned above.

%%%%%%%%%%%%%%%%%%%%%%%%%%%%%

\section{Augmented rook board}

%%%%%%%%%%%%%%%%%%%%%%%%%%%%%

In this section, we review the rook theory model on augmented rook boards
defined by Miceli and Remmel~\cite{MR}.
We consider two sequences of nonnegative integers of length $n$,
$\mathcal{A}=\{a_i\}_{i=1}^{n}$ and $\mathcal{B}=\{b_i\}_{i=1}^{n}$, and two
functions $sgn, \overline{sgn}:[n]\rightarrow\{1, -1\}$. Let
$A_i =a_1+a_2+\cdots +a_i$ be the $i$-th partial sum of the $a_i$'s and
$B=B(b_1, b_2,\cdots, b_n)$. The \emph{augmented rook board}
$\mathcal{B}^{\mathcal{A}}$ is constructed by adding $A_i$ cells 
on top of $b_i$ in the $i$-th column for $i=1,\dots, n$.
Note that $\mathcal{B}^{\mathcal{A}}$ can be considered as the Ferrers board
$B(b_1+A_1, b_2+A_2,\dots, b_n+A_n)$. 
We refer to the part of the board corresponding to the $b_i$'s as the
\emph{base part} of $\mathcal{B}^{\mathcal{A}}$ and the part corresponding
to the $a_i$'s as the \emph{augmented part} of $\mathcal{B}^{\mathcal{A}}$.
Moreover, for each column $i$, $i=1,\dots, n$, we refer to the cells in
rows $b_i +1,\dots, b_i+a_1$ as the $a_1$-st part, the cells in rows
$b_i+a_1+1,\dots , b_i+a_1+a_2$ as the $a_2$-nd part, and the cells in
rows $b_i+a_{j-1}+1,\dots, b_i+a_{j-1}+a_j$ as the $a_j$-th part, in general.
Figure~\ref{fig:B^A} is an example of an augmented rook board for
$\B=(1,2,2,3)$ and $\A=(1,2,1,2)$. In the figure, the cells corresponding
to the $a_i$-th part are filled with $i$'s. 
\begin{figure}[ht]
$$\begin{picture}(40,100)(0,0)
\multiput(0,0)(0,10){3}{\line(1,0){40}}
\multiput(10,30)(0,10){4}{\line(1,0){30}}
\multiput(20,70)(0,10){1}{\line(1,0){20}}
\multiput(30,80)(0,10){3}{\line(1,0){10}}
\multiput(0,0)(0,10){1}{\line(0,1){20}}
\multiput(10,0)(0,10){1}{\line(0,1){60}}
\multiput(20,0)(0,10){1}{\line(0,1){70}}
\multiput(30,0)(10,0){2}{\line(0,1){100}}
\thicklines \linethickness{1.3pt}
\multiput(0,10)(0,10){2}{\line(1,0){10}}
\multiput(0,10)(10,0){2}{\line(0,1){10}}
\multiput(10,30)(0,10){1}{\line(0,1){30}}
\multiput(20,30)(10,0){1}{\line(0,1){40}}
\multiput(30,30)(0,10){1}{\line(0,1){70}}
\multiput(40,40)(10,0){1}{\line(0,1){60}}
\multiput(10,30)(0,10){1}{\line(1,0){20}}
\multiput(10,40)(0,10){3}{\line(1,0){30}}
\multiput(20,70)(0,10){1}{\line(1,0){20}}
\multiput(30,80)(0,10){3}{\line(1,0){10}}
\put(2,11){$1$}
\put(12,31){$1$}
\put(22,31){$1$}
\put(32,41){$1$}
\put(12,41){$2$}
\put(12,51){$2$}
\put(22,41){$2$}
\put(22,51){$2$}
\put(32,51){$2$}
\put(32,61){$2$}
\put(22,61){$3$}
\put(32,71){$3$}
\put(32,81){$4$}
\put(32,91){$4$}
\end{picture}$$
\caption{An example of $\B^{\A}$ for $\B=(1,2,2,3)$ and $\A=(1,2,1,2)$.}
\label{fig:B^A}
\end{figure}
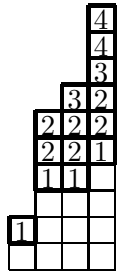

Next we define the rook cancellation of a rook placement in
$\mathcal{B}^{\mathcal{A}}$. We consider placements $P$ of rooks in
$\mathcal{B}^{\mathcal{A}}$ with at most one rook in each column.
The leftmost rook of $P$ will cancel all the cells in the columns to its
right which correspond to the $a_s$-th part of that column of highest
index $s$. In general, each rook cancels the cells in the columns to its
right which correspond to the $a_s$-th part of that column where $s$ is
the highest index such that the cells of $a_s$-th part of that column
have not been cancelled by any rook to its left. We say that a rook
placement is \emph{nonattacking} if 
\begin{itemize}
 \item[(i)] there is at most one rook in each column, and 
 \item[(ii)] no rook lies in a cell which has been cancelled by a
rook to its left. 
\end{itemize}
Let $\mathcal{N}_k ^{\mathcal{A}}(\mathcal{B}^{\mathcal{A}})$ denote the set of
nonattacking rook placements of $k$ rooks in $\mathcal{B}^{\mathcal{A}}$. Define 
\begin{equation}
 r_k ^{\mathcal{A}} (\mathcal{B}^{\mathcal{A}}, sgn, \overline{sgn})=
\sum_{P\in \mathcal{N}_k ^{\mathcal{A}}(\mathcal{B}^{\mathcal{A}})}
w_{sgn, \overline{sgn}, \mathcal{B}^{\mathcal{A}}}(P),
\end{equation}
where 
$$
w_{sgn, \overline{sgn}, \mathcal{B}^{\mathcal{A}}}(P)=
\prod_{c\in P}w_{sgn, \overline{sgn}, \mathcal{B}^{\mathcal{A}},P}(c),
$$
for
$$
w_{sgn, \overline{sgn}, \mathcal{B}^{\mathcal{A}},P}(c)=
\left\{\begin{array}{ll}sgn(i),&
\text{if $c$ is in column $i$ and in the base part of
$\mathcal{B}^{\mathcal{A}}$,}\\
-\overline{sgn}(s),& \text{ if $c$ is in the $a_s$-th part of the
augmented part of $\mathcal{B}^{\mathcal{A}}$.}
\end{array}\right.
$$
Then Miceli and Remmel \cite{MR} proved the following theorem.

\begin{theorem}\cite[Theorem 3.1]{MR}\label{thm:MR}
 Suppose $\mathcal{A}=(a_1,\dots, a_n)$ and $\mathcal{B}=(b_1,\dots, b_n)$
are two sequences of nonnegative integers and
$sgn:\{1,\dots, n\}\rightarrow \{1, -1\}$
 and $\overline{sgn}:\{1,\dots, n\} \rightarrow \{1, -1\}$ are two
sign functions. Then, 
 \begin{equation}
  \prod_{i=1}^{n}(z+sgn(i)b_i)=
\sum_{k=0}^n r_{n-k}^{\mathcal{A}}(\mathcal{B}^{\mathcal{A}}, sgn, \overline{sgn})
\prod_{j=1}^k (z+ \sum_{s\le j}\overline{sgn}(s)a_s).
 \end{equation}
\end{theorem}

We refer to the paper of Miceli and Remmel~\cite{MR} for the detailed
proof of Theorem~\ref{thm:MR}, but we introduce the extended augmented
board $\B_z ^\A$ for later use. 
Given two sequences of nonnegative integers $\A$, $\B$ and a nonnegative
integer $z$, the board $\B_z ^\A$ consists of three parts.
We start with the board $\B ^\A$ which we refer to as the \emph{upper part}
of $\B _z ^\A$. Within the upper part, the cells corresponding to the
board $B=B(b_1, \dots, b_n)$ will be called the \emph{base part} of
$\B_z ^\A$ and the part corresponding to the $a_i$'s will be called the
\emph{upper augmented part} of $\B_z ^\A$. Directly below $\B ^\A$,
we attach $n$-columns of length $z$ which will be referred to as the
\emph{$z$-part} of $\B _z ^\A$. Finally, directly below the $z$-part,
we place the reflected Ferrers board $B(A_1, \dots, A_n)$ which will be called 
the \emph{lower augmented part} of $\B_z ^\A$. We call the line separating
the $z$-part and the upper part $\B^\A$ the \emph{high bar} and the line
separating the lower augmented part and the $z$-part the \emph{low bar}. 
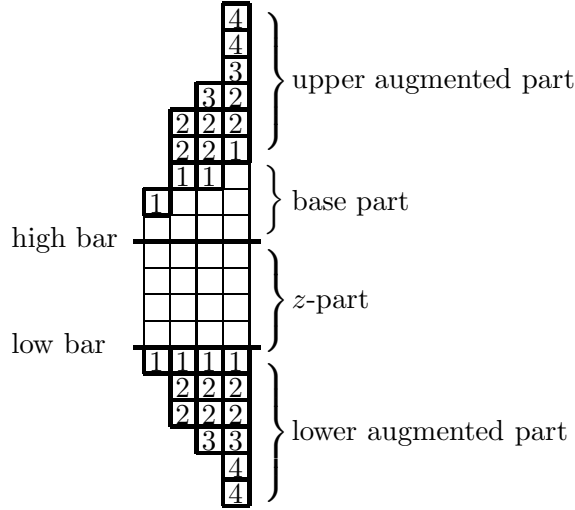
\begin{figure}[ht]
$$\begin{picture}(100,190)(30,0)
\multiput(80,0)(0,10){2}{\line(1,0){10}}
\multiput(70,20)(0,10){1}{\line(1,0){20}}
\multiput(60,30)(0,10){2}{\line(1,0){30}}
\multiput(50,50)(0,10){8}{\line(1,0){40}}
\multiput(60,130)(0,10){3}{\line(1,0){30}}
\multiput(70,160)(0,10){1}{\line(1,0){20}}
\multiput(80,170)(0,10){3}{\line(1,0){10}}
\multiput(50,50)(10,0){1}{\line(0,1){70}}
\multiput(60,30)(10,0){1}{\line(0,1){120}}
\multiput(70,20)(10,0){1}{\line(0,1){140}}
\multiput(80,0)(10,0){2}{\line(0,1){190}}
\thicklines \linethickness{1.3pt}
\multiput(46,60)(0,40){2}{\line(1,0){48}}
\thicklines \linethickness{1.28pt}
%%%%%%%%%%%%%%%%%%%%%%%%%%%%%%%%%%%%%%%%%%%%%%%%%%%%%%
\multiput(50,110)(0,10){2}{\line(1,0){10}}
\multiput(50,110)(10,0){2}{\line(0,1){10}}
\multiput(60,120)(0,10){1}{\line(0,1){30}}
\multiput(70,120)(10,0){1}{\line(0,1){40}}
\multiput(80,120)(0,10){1}{\line(0,1){70}}
\multiput(90,130)(10,0){1}{\line(0,1){60}}
\multiput(60,120)(0,10){1}{\line(1,0){20}}
\multiput(60,130)(0,10){3}{\line(1,0){30}}
\multiput(70,160)(0,10){1}{\line(1,0){20}}
\multiput(80,170)(0,10){3}{\line(1,0){10}}
%%%%%%%%%%%%%%%%%%%%%%%%%%%%%%%%%%%%%%%%%%%%%%%%%%%%
\multiput(50,50)(0,10){1}{\line(1,0){40}}
\multiput(60,30)(0,10){2}{\line(1,0){30}}
\multiput(70,20)(0,10){1}{\line(1,0){20}}
\multiput(80,0)(0,10){2}{\line(1,0){10}}
\multiput(50,50)(10,0){1}{\line(0,1){10}}
\multiput(60,30)(10,0){1}{\line(0,1){30}}
\multiput(70,20)(10,0){1}{\line(0,1){40}}
\multiput(80,0)(10,0){2}{\line(0,1){60}}
%%%%%%%%%%%%%%%%%%%%%%%%%%%%%%%%%%%%%%%%%%%%%
\put(73, 158){$\left.\begin{array}{c}\quad\\\\\\\\ \end{array}\right\}$}
\put(106, 158){upper augmented part}
\put(73, 113){$\left.\begin{array}{c}\quad\\\\ \end{array}\right\}$}
\put(106,112){base part}
\put(73, 75){$\left.\begin{array}{c}\quad\\\\\\ \end{array}\right\}$}
\put(106,75){$z$-part}
\put(73, 25){$\left.\begin{array}{c}\quad\\\\\\\\ \end{array}\right\}$}
\put(106, 25){lower augmented part}
\put(0, 58){low bar}
\put(0, 98){high bar}
\put(52,111){$1$}
\put(62,121){$1$}
\put(72,121){$1$}
\put(82,131){$1$}
\put(62,131){$2$}
\put(62,141){$2$}
\put(72,131){$2$}
\put(72,141){$2$}
\put(82,141){$2$}
\put(82,151){$2$}
\put(72,151){$3$}
\put(82,161){$3$}
\put(82,171){$4$}
\put(82,181){$4$}
%%%%%%%%%%%%%%%%%%%%%%%%%
\put(52,51){$1$}
\put(62,51){$1$}
\put(72,51){$1$}
\put(82,51){$1$}
\put(62,41){$2$}
\put(62,31){$2$}
\put(72,41){$2$}
\put(72,31){$2$}
\put(82,41){$2$}
\put(82,31){$2$}
\put(72,21){$3$}
\put(82,21){$3$}
\put(82,1){$4$}
\put(82,11){$4$}
\end{picture}$$
\caption{An example of an extended augmented rook board}\label{fig:extaugbd}
\end{figure}
Figure~\ref{fig:extaugbd} shows an example of an extended augmented
rook board for $\B=(1,2,2,3)$, $\A=(1,2,1,2)$ and $z=4$. The cells in
the upper augmented part and lower augmented part 
are filled with $i$ if they correspond to the $a_i$-part in each column. 

In \cite{MR}, Miceli and Remmel also prove the following $q$-analogue of
Theorem~\ref{thm:MR}.

\begin{theorem}\cite[Theorem 4.2]{MR}\label{thm:MRq}
 Suppose $\A=(a_1,\dots, a_n)$ and $\B=(b_1, \dots, b_n)$ are two sequences
of nonnegative integers and $sgn:\{1,\dots, n\}\rightarrow \{1, -1\}$ and 
 $\overline{sgn}:\{1,\dots, n\}\rightarrow \{1, -1\}$ are two sign functions.
Then,
 \begin{equation}\label{eqn:MRq}
  \prod_{i=1}^n [z+sgn(i)b_i]_q =
\sum_{k=0}^n R_{n-k}^\A (\mathcal{B}^{\mathcal{A}}, sgn, \overline{sgn},q)
\prod_{j=1}^k ([z+\sum_{s\le j}\overline{sgn}(s)a_s]_q).
 \end{equation}
\end{theorem}

Here, $R_{n-k}^\A (\mathcal{B}^{\mathcal{A}}, sgn, \overline{sgn},q)$ is a
specially defined $q$-analogue of
$r_{n-k}^{\mathcal{A}}(\mathcal{B}^{\mathcal{A}}, sgn, \overline{sgn})$.
The proof of Theorem~\ref{thm:MRq} can be done by assigning $q$-weights
to the cells of the extended augmented rook board $\B_z ^\A$.
We will see the detailed proof in the process of 
proving an elliptic extension of \eqref{eqn:MRq}.

%%%%%%%%%%%%%%%%%%%%%%%%%%%%%%%%%%%%%%%%%%%%%%%%%%%%%%%%%%%%%%%%%%%%%%%%5

\section{Elliptic extension}

%%%%%%%%%%%%%%%%%%%%%%%%%%%%%%%%%%%%%%%%%%%%%%%%%%%%%%%%%%%%%%%%%%%%%%%%5

In this section, we derive an elliptic extension of Theorem~\ref{thm:MRq}.
We first briefly explain the notion of an elliptic function.
The multiplicative notation we adopt is common in the context of
elliptic hypergeometric series, cf.\ \cite[Chapter~11]{GRhyp}.

A complex function is called \emph{elliptic}
if it is doubly periodic and meromorphic. 
Since elliptic functions can be built from quotients of theta functions,
we define and use theta functions to construct elliptic functions. 

Define a \emph{modified Jacobi theta function} by 
$$
\theta(x;p)=\prod_{j\ge 0} (1-p^j x)(1-p^{j+1}/x),\quad \theta(x_1,\dots, x_m ;p)
=\prod_{k=1}^m \theta(x_k ;p),
$$
where $x, x_1,\dots, x_m\ne 0$ and $|p|< 1$. Note that this definition is 
based on Jacobi's triple product identity. We also define an 
\emph{elliptic shifted factorial} analogue of the $q$-shifted factorial by 
$$(a;q,p)_n =\begin{cases}
             \prod_{k=0}^{n-1} \theta(a q^k ;p), & n=1,2,\dots, \\
             1, & n=0,\\
             1/\prod_{k=0}^{-n-1}\theta(a q^{n+k};p), & n=-1, -2, \dots,
             \end{cases}$$
and let 
$$(a_1,a_2,\dots, a_m;q,p)_n=\prod_{k=1}^n (a_k;q,p)_n,$$
where $a, a_1,\dots, a_m\ne 0$. The parameter $q$ is called the \emph{base} and 
$p$ is called the \emph{nome}. Note that $\theta(x;0)=1-x$ and thus
$(a;q,0)_n =(a;q)_n$.
Among many identities involving the Jacobi theta functions and the
elliptic shifted factorials
(see \cite[(11.2.42)--(11.2.60)]{GRhyp}), the following
\emph{addition formula} is crucial 
in the theory of elliptic hypergeometric series 
\begin{equation}\label{eqn:theta_addition}
\theta(xy, x/y, uv, u/v;p)-\theta(xv, x/v, uy, u/y;p)=
\frac{u}{y}\theta(yv, y/v, xu, x/u;p).
\end{equation}

We next define the \emph{elliptic weights} which are an elliptic extension
of the $q$-weights.
Let 
\begin{equation}\label{eqn:sw}
w_{a,b;q,p}(k)=\frac{\theta(aq^{2k+1}, bq^k, a q^{k-2}/b;p)}
{\theta(a q^{2k-1}, bq^{k+2}, a q^k/b;p)}q, 
\end{equation}
and 
\begin{equation}\label{eqn:bw}
W_{a,b;q,p}(k)=\frac{\theta(a q^{2k+1},bq,bq^2,aq^{-1}/b, a/b;p)}
{\theta(aq, bq^{k+1},bq^{k+2},aq^{k-1}/b, aq^k /b;p)}q^k.
\end{equation}
Note that for a positive integer $k$, the elliptic weights are related as
$$
W_{a,b;q,p}(k)=\prod_{j=1}^k w_{a,b;q,p}(j),
~\text{ or }~
w_{a,b;q,p}(k)=\frac{W_{a,b;q,p}(k)}{W_{a,b;q,p}(k-1)}.
$$
We also have 
\begin{subequations}
\begin{equation}
w_{a,b;q,p}(k+n)=w_{aq^{2k}, bq^k;q,p}(n)
\end{equation}
and 
\begin{equation}\label{eqn:Wid}
W_{a,b;q,p}(k+n)=W_{a,b,q,p}(k)W_{aq^{2k},bq^k;q,p}(n).
\end{equation}
\end{subequations}
\begin{remark}
If we let $p\to 0$, $a\to 0$ and $b\to 0$, in this order
(or, $p\to 0$, $b\to 0$ and $a\to \infty$ in this order),
then we recover the $q$-weights 
$$
w_{0,0;q,0}(k)=q ~\text{ and }~ W_{0,0;q,0}(k)=q^k.
$$
We can easily verify that the weights $w_{a,b;q,p}(k)$ and $W_{a,b;q,p}(k)$  
are indeed elliptic. More precisely, if we let $q=e^{2\pi i\sigma}$,
$p=e^{2\pi i\tau}$, $a=q^\alpha$, $b=q^{\beta}$
for $\sigma, \tau, \alpha, \beta \in \mathbb{C}$, then $w_{a,b;q,p}(k)$ is 
periodic in $\alpha$ with periods $\sigma^{-1}$ and $\tau \sigma^{-1}$. 
As a function in $\beta$ (or $k$) the same applies to $w_{a,b;q,p}(k)$.
\end{remark}

We define an \emph{elliptic number} of $z$ by
\begin{equation}\label{eqn:elptnumber}
[z]_{a,b;q,p}=\frac{\theta(q^z, a q^z, b q^2, a/b;p)}
{\theta(q, aq, b q^{z+1}, aq^{z-1}/b;p)}. 
\end{equation}
It can be verified that the elliptic numbers satisfy 
\begin{subequations}
\begin{equation}\label{eqn:zrec}
[z]_{a,b;q,p}=[z-1]_{a,b;q,p}+W_{a,b;q,p}(z-1)
\end{equation}
using the addition formula \eqref{eqn:theta_addition} and more generally, 
\begin{equation}\label{eqn:zrecgen}
[z]_{a,b;q,p}=[y]_{a,b;q,p}+W_{a,b;q,p}(y)[z-y]_{a q^{2y},bq^y ;q,p}
\end{equation}
which reduces to \eqref{eqn:zrec} for $y=z-1$.
\end{subequations} 
Hence
\begin{subequations}
\begin{equation}\label{eqn:zsum}
[z]_{a,b;q,p}=1+W_{a,b;q,p}(1)+\cdots +W_{a,b;q,p}(z-1),
\end{equation}
where $W_{a,b;q,p}(0)=1$.
From \eqref{eqn:zrec} we can also deduce
\begin{equation}\label{eqn:-zsum}
-[-z]_{a,b;q,p}=W_{q,b;q,p}(-1)+W_{a,b;q,p}(-2)+\cdots +W_{a,b;q,p}(-z+1)+W_{q,b;q,p}(-z).
\end{equation}
\end{subequations}
We remark that the definitions of the elliptic weights $w_{a,b;q,p}(k)$,
$W_{a,b;q,p}(k)$ and the elliptic numbers $[z]_{a,b;q,p}$ originate
from the elliptic binomial coefficients 
\begin{equation}\label{eqn:elptbinom}
\begin{bmatrix}n\\k\end{bmatrix}_{a,b;q,p}:=
\frac{(q^{1+k},aq^{1+k},bq^{1+k},aq^{1-k}/b;q,p)_{n-k}}
{(q,aq,bq^{1+2k},aq/b;q,p)_{n-k}},
\end{equation}
as defined by one of us in \cite{Schl1}.
It is easy to see that the expression in \eqref{eqn:elptbinom} reduces
to the usual $q$-binomial coefficient
if one lets $p\to 0$, $a\to 0$ and $b\to 0$.
The elliptic binomial coefficients admit a combinatorial interpretation
in terms of \emph{weighted lattice paths} in $\mathbb{Z}^2$.
Consider lattice paths $P$ from $(0,0)$ to $(k, n-k)$ consisting of east
and north steps only. For each horizontal step $(s-1,t)\to (s,t)$,
assign the weight $W_{a q^{s-1}, bq^{2s-2};q,p}(t)$ and $1$ to each
vertical step. If we define the weight of path $P$, denoted by $wt(P)$,
to be the product of all the weights of the respective steps of the path, 
then the elliptic binomial coefficient is the weight generating function
of all the paths from $(0,0)$ to $(k,n-k)$, i.e.,
$$
\begin{bmatrix}n\\k\end{bmatrix}_{a,b;q,p}=
\sum_{P\in \mathcal P ((0,0)\to (k,n-k))}wt(P),
$$
where $\mathcal P (A \to B)$ is the set of all the lattice paths from
$A$ to $B$. The proof of this identity is based on the recurrence relation 
\begin{equation}\label{eqn:recw}
\begin{bmatrix}n+1\\k\end{bmatrix}_{a,b;q,p}=
\begin{bmatrix}n\\k\end{bmatrix}_{a,b;q,p}
+\begin{bmatrix}n\\k-1\end{bmatrix}_{a,b;q,p}
\,W_{aq^{k-1},bq^{2k-2};q,p}(n+1-k)
\qquad \text{for $n,k\in\mathbb N \cup \{0\}$}
\end{equation}
with the initial conditions 
\begin{equation}
\begin{bmatrix}n\\0\end{bmatrix}_{a,b;q,p}=1,\qquad
\begin{bmatrix}n\\k\end{bmatrix}_{a,b;q,p}=0
\qquad\text{for\/ $n\in\mathbb N \cup \{0\}$, and\/
$k\in-\mathbb N$ or $k>n$}.
\end{equation}
If we let $k=1$ in \eqref{eqn:recw}, then we get 
$$\begin{bmatrix}n+1\\1\end{bmatrix}_{a,b;q,p}=
\begin{bmatrix}n\\1\end{bmatrix}_{a,b;q,p}
+
\,W_{a,b;q,p}(n)$$
which, by comparing to \eqref{eqn:zrec}, shows that the elliptic number
$[n]_{a,b;q,p}$ is equal to $\begin{bmatrix}n\\1\end{bmatrix}_{a,b;q,p}$.

Now we are ready to construct an elliptic extension of Theorem~\ref{thm:MRq}.
%%%%%%%%%%%%%%%%%%%%%%%%%%%%%%%%%%%%%%%%%%%%%%%%%%%%%%%%%%%%%%%%%%%%%%%%%%%%%%%
We remark that the way to assign the weights to the cells in $\B_z ^\A$
is similar to the way done in \cite{MR}. The main idea is to extend the
$q$-number (or $p,q$-number) to the elliptic number.

Consider two sequences $\A=(a_1,\dots, a_n)$, $\B=(b_1,\dots, b_n)$ and 
two sign functions $sgn:\{1,\dots, n\}\rightarrow \{1, -1\}$,
$\overline{sgn}:\{1,\dots, n\}\rightarrow \{1, -1\}$. We let
$\overline{A}_s:= \sum_{i=1}^s \overline{sgn}(i)a_i$.
We assign the elliptic weight $M_{a,b;q,p}(\B_z ^\A, sgn, \overline{sgn};c)$
to each cell $c \in \B_z ^\A$ according to the following scheme. 
\begin{itemize}
 \item[(i)] For each $i$, the cells $c$ in the $i$-th column of the $z$-part
of $\B_z ^\A$ have weights 
 $$1, W_{a,b;q,p}(1), W_{a,b;q,p}(2), \dots, W_{a,b;q,p}(z-1),$$ reading 
 from bottom to top. 
 \item[(ii)] For each $i$, the cells $c$ in the $i$-th column of the base part
of $\B_z ^\A$ have weights 
 $$\left\{ \begin{array}{ll}-1, -W_{a,b;q,p}(1), -W_{a,b;q,p}(2),
\dots,- W_{a,b;q,p}(b_i-1),& \text{ if } sgn(i)=-1 \\ 
 W_{a,b;q,p}(-1), W_{a,b;q,p}(-2), \dots, W_{a,b;q,p}(-b_i), &
\text{ if } sgn(i)=1, \end{array}\right. $$
 reading from bottom to top.
 \item[(iii)] For each $i$, we assign the elliptic weights to the
cells $c$ in the $i$-th column of the lower augmented part as follows.
First, to the cells in the $a_1$-st part of column $i$, we assign the weights 
 $$\left\{ \begin{array}{ll}-1, -W_{a,b;q,p}(1), -W_{a,b;q,p}(2),
\dots,- W_{a,b;q,p}(a_1-1),& \text{ if } \overline{sgn}(1)=-1 \\ 
 W_{a,b;q,p}(-1), W_{a,b;q,p}(-2), \dots, W_{a,b;q,p}(-a_1), &
\text{ if } \overline{sgn}(1)=1, \end{array}\right. $$
 reading from top to bottom. Note that the sum of the weights of the
cells in $a_1$-st part becomes $-[-\overline{sgn}(1)a_1]_{a,b;q,p}$. 
 In the case when $sgn(1)=1$, we used the identity \eqref{eqn:-zsum}.
 Suppose that we have assigned the weights to cells in the $a_j$-th part
of column $i$ in the lower augmented part for $j=1,\dots, s$ so that 
 the sum of the weights of cells that lie in the $a_j$-th part of column
$i$ for $j\le s$ is $-[-\overline{A}_s]_{a,b;q,p}$. 
 Then we assign the weights to the cells in the $a_{s+1}$-st part of column $i$
 in the lower augmented part according to the following cases. 
 \begin{description}
 \item[Case 1] $0\le \overline{A}_s \le \overline{A}_{s+1}$\\
 In this case, the weights of the cells in the $a_{s+1}$-st part are,
reading from top to bottom,
 $$W_{a,b;q,p}(-\overline{A}_s-1), W_{a,b;q,p}(-\overline{A}_s -2),
\dots, W_{a,b;q,p}(-\overline{A}_{s+1}).$$
 \item[Case 2] $0\le \overline{A}_{s+1} < \overline{A}_s$\\
 Then the weights of the cells in the $a_{s+1}$-st part are
 $$-W_{a,b;q,p}(-\overline{A}_s), -W_{a,b;q,p}(-\overline{A}_s +1),
\dots, -W_{a,b;q,p}(-\overline{A}_{s+1}-1),$$
 reading from top to bottom.
 \item[Case 3] $\overline{A}_{s+1} < 0\le \overline{A}_s$\\
 In this case, the weights of the cells in the $a_{s+1}$-st part are 
 \begin{multline*}
 \qquad\qquad -W_{a,b;q,p}(-\overline{A}_s ), -W_{a,b;q,p}(-\overline{A}_s +1),
\dots, -W_{a,b;q,p}(-1),-1,\\
 -W_{a,b;q,p}(1), \dots, -W_{a,b;q,p}(-\overline{A}_{s+1} -1),
 \end{multline*}
 reading from top to bottom.
 \item[Case 4] $\overline{A}_{s+1}\le \overline{A}_s\le 0$\\
 Then the weights of the cells in the $a_{s+1}$-st part are,
reading from top to bottom,
 $$-W_{a,b;q,p}(-\overline{A}_s), -W_{a,b;q,p}(-\overline{A}_s +1),
\dots, -W_{a,b;q,p}(-\overline{A}_{s+1} -1).$$
 \item[Case 5] $\overline{A}_s  < \overline{A}_{s+1}\le 0$\\
 In this case, the weights of the cells in the $a_{s+1}$-st part are 
 $$ W_{a,b;q,p}(-\overline{A}_s -1), W_{a,b;q,p}(-\overline{A}_s -2),
\dots, W_{a,b;q,p}(-\overline{A}_{s+1} ),$$
 reading from top to bottom.
 \item[Case 6] $\overline{A}_s \le 0  < \overline{A}_{s+1}$\\
 Then, the weights of the cells in the $a_{s+1}$-st part are
 \begin{multline*}
 \qquad\qquad W_{a,b;q,p}(-\overline{A}_s -1), W_{a,b;q,p}(-\overline{A}_s -2),
\dots, W_{a,b;q,p}(1),1,\\
 W_{a,b;q,p}(-1), W_{a,b;q,p}(-2),\dots, W_{a,b;q,p}(-\overline{A}_{s+1}),
 \end{multline*}
 reading from top to bottom.
 \end{description}
 The weights are defined so that the sum of the weights of cells
that lie in the $a_j$-th part of column $i$ for $j\le s+1$ 
 also becomes $-[-\overline{A}_{s+1}]_{a,b;q,p}$. 
 \item[(iv)] For each $i$, the weight of the cell in the $r$-th row
of the $i$-th column of the upper augmented part, reading from the bottom, 
is equal to $-1$ times the weight of the cell in the $r$-th row of $i$-th
column of the lower augmented board, reading from the top. That is, the weight
of the cell in the upper augmented part of column $i$ is just the negative
of the weight of its corresponding cell in the lower augmented part.
 \end{itemize}
Suppose that $P\in \N_k ^{\A}(\B^\A)$ has rooks in cells $c_1,\dots, c_k$.
Then we set 
\begin{equation}\label{eqn:Mdef}
M_{a,b;q,p}(\B^\A, sgn, \overline{sgn};P)=\prod_{i=1}^k
M_{a,b;q,p}(\B_z ^\A, sgn, \overline{sgn};c_i),
\end{equation}
and we define 
\begin{equation}\label{eqn:MRdef}
M\!R_k ^\A (a,b;q,p; \B^\A, sgn, \overline{sgn})=
\sum_{P\in \N_k ^{\A}(\B^\A)}M_{a,b;q,p}(\B^\A, sgn, \overline{sgn};P),
\end{equation}
with $M\!R_0 ^\A (a,b;q,p; \B^\A, sgn, \overline{sgn})=1$.
Finally we define 
\begin{align}\label{eqn:Rdef}
 &R_{n-k}^\A (a,b;q,p; \B^\A, sgn, \overline{sgn})\notag\\
 &=\frac{\prod_{s\le k}W_{a,b;q,p}(- \overline{A}_s)}
{\prod_{i=1}^n W_{a,b;q,p}(-sgn(i)b_i) }
 M\!R_{n-k} ^\A (a,b;q,p; \B^\A, sgn, \overline{sgn}).\qquad
\end{align}
Then we can prove the following elliptic extension of Theorem~\ref{thm:MRq}.

\begin{theorem}\label{thm:elptMR}
Suppose that two sequences of nonnegative integers $\A =(a_1, \dots, a_n)$,
$\B=(b_1, \dots, b_n)$ and two sign functions
$sgn:\{1,\dots, n\}\rightarrow \{1, -1\}$,  
$\overline{sgn}:\{1,\dots, n\}\rightarrow \{1, -1\}$ are given. Then, 
\begin{align}\label{eqn:elptMR}
& \prod_{i=1}^n [z+sgn(i)b_i]_{a q^{-2 sgn(i)b_i},b q^{-sgn(i)b_i};q,p}\notag\\
&=\sum_{k=0}^n R_{n-k}^\A (a,b;q,p; \B^\A, sgn, \overline{sgn})
 \prod_{1 \le j\le k}[z+\overline{A}_j]_{aq^{-2\overline{A}_j},bq^{-\overline{A}_j};q,p}.
\end{align}
\end{theorem}

\begin{proof}
We first set the cancellation scheme of the rook placements in the
extended augmented board $\B_z ^\A$. We consider placements of $n$ rooks
in $\B_z ^\A$ where there is exactly one rook in each column. If a rook
is placed above the high bar in the $j$-th column, then it cancels all
the cells in columns $j+1, j+2, \dots, n$, in both the upper and lower
augmented parts, which belong to the $a_i$-th part of the highest subscript
in that column which are not cancelled by a rook to the left of column $j$.
A rook which is placed below the high bar does not cancel anything.
We say that a rook placement is \emph{nonattacking} if no rook lies in
a cell which is cancelled by a rook to its left. We denote the set of all
nonattacking rook placements of $n$ rooks
in $\B_z ^\A$ by $\N _n ^\A (\B_z ^\A )$.  

We set, for $Q\in \N_n ^\A (\B _z^\A)$ having rooks in cells $c_1,\dots, c_n$, 
\begin{equation}
M_{a,b;q,p}(\B_z^\A, sgn, \overline{sgn};Q)=\prod_{i=1}^n
M_{a,b;q,p}(\B_z ^\A, sgn, \overline{sgn};c_i).
\end{equation}
By computing the sum
$$WT(a,b;q,p)=\sum_{Q\in \N_n ^\A (\B _z ^\A)}
M_{a,b;q,p}(\B_z^\A, sgn, \overline{sgn};Q)$$
in two different ways, we first prove the following product formula.
\begin{align}\label{eqn:elptMRpf}
&\!\left(\prod_{i:sgn(i)=1} [z]_{a,b;q,p}-[-b_i]_{a,b;q,p}\right)
\left(\prod_{i:sgn(i)=-1}[z]_{a,b;q,p}-[b_i]_{a,b;q,p}\right)\notag\\
& =\sum_{k=0}^n M\!R_{n-k}^\A (a,b;q,p; \B^\A, sgn, \overline{sgn}) 
 \prod_{j=1}^k ([z]_{a,b;q,p}-[-\overline{A}_j]_{a,b;q,p}).
\end{align}
We place $n$ rooks column-wise, starting from the leftmost column. 
Note that the weights coming from the cases of placing the first rook 
in the upper augmented part corresponding to $a_1$-part will be cancelled
with the weights coming from placing the rook in the corresponding cells
in the lower augmented part.
So, starting from the bottom cell in the $z$-part, the possible weights are 
$$ 1, W_{a,b;q,p}(1),\dots, W_{a,b;q,p}(z-1), -1, -W_{a,b;q,p}(1), -W_{a,b;q,p}(2),
\dots, -W_{a,b;q,p}(b_1-1)$$
if $sgn(1)=-1$, and 
$$ 1, W_{a,b;q,p}(1),\dots, W_{a,b;q,p}(z-1), W_{a,b;q,p}(-1), W_{a,b;q,p}(-2),
\dots, W_{a,b;q,p}(-b_1 )$$
if $sgn(1)=1$. Hence the weight sum coming from the possible rook placements
in the first column is 
$[z]_{a,b;q,p}-[b_1]_{a,b;q,p}$ if $sgn(1)=-1$, and $[z]_{a,b;q,p}-[-b_1]_{a,b;q,p}$
if $sgn(1)=1$. 
In general, in each column $i$, the weights of the upper and lower
augmented parts cancel each other, and 
by considering the weights coming from the $z$-part and the base part, 
we get $[z]_{a,b;q,p}-[b_i]_{a,b;q,p}$ if $sgn(i)=-1$, and
$[z]_{a,b;q,p}-[-b_i]_{a,b;q,p}$ if $sgn(i)=1$.
By multiplying the factors coming from the columns $i=1,\dots, n$,
we get the left-hand side of $($\ref{eqn:elptMRpf}$)$.

On the other hand, we consider placing $n-k$ rooks above the high bar
and $k$ rooks below it. Fix a rook placement $P\in \N_{n-k}^\A(\B^\A)$
and extend it to $\N_n ^\A (\B_z ^\A)$ by placing $k$ more rooks below
the high bar, and compute the weight sum coming from placing $k$ more rooks. 
Note that there are $k$ columns containing no rooks. 
Let us say the $l$-th column is the first empty column. Then the lower
augmented part consists of $a_1+\cdots +a_l$ cells, 
but the $l-1$ rooks to the left of the $l$-th column
cancel the cells in $a_l, a_{l-1}, \dots, a_2$ parts, and so it has only
$a_1$ part uncancelled. 
In general, the $i$-th empty column from the left would have $a_1+\cdots +a_i$ 
uncancelled cells. Note that due to the way of assigning the weights
to the cells in the lower augmented part, the sum of the weights in the
cells $a_1,\dots, a_i$ equals to $-[-\overline{A}_i]_{a, b;q,p}$. 
Thus the sum of the weights placing a rook in the $i$-th available column
becomes $[z]_{a,b;q,p}-[-\overline{A}_i]_{a, b;q,p}$.  
Recall that placing a rook below the high bar does not cancel any cells.
Hence,
\begin{align*}
WT(a,b;q,p)
&= \sum_{k=0}^n \sum_{P\in \N_{n-k}^\A (\B^\A)}M_{a,b;q,p}(\B^\A , sgn,\overline{sgn};P)
\prod_{1\le j\le k}([z]_{a,b;q,p}-[-\overline{A}_j]_{a, b;q,p})\\
 &=\sum_{k=0}^n M\!R_{n-k} ^\A (a,b;q,p; \B^\A, sgn, \overline{sgn})
\prod_{1\le j\le k}([z]_{a,b;q,p}-[-\overline{A}_j]_{a, b;q,p})
\end{align*} 
which gives the right hand side of \eqref{eqn:elptMRpf}.
Recall that we have the following identity in \eqref{eqn:zrecgen}
$$[z]_{a,b;q,p}=[y]_{a,b;q,p}+W_{a,b;q,p}(y)[z-y]_{aq^{2y},bq^y;q,p}.$$
Using this identity, the factors in the left hand side of \eqref{eqn:elptMRpf}
can be rewritten as 
$$[z]_{a,b;q,p}-[b_i]_{a,b;q,p}=W_{a,b;q,p}(b_i)[z-b_i]_{a q^{2 b_i},b q^{b_i};q,p}$$ 
in the case when $sgn(i)=-1$ and 
$$[z]_{a,b;q,p}-[-b_i]_{a,b;q,p}=W_{a,b;q,p}(-b_i)[z+b_i]_{a q^{-2 b_i},b q^{-b_i};q,p}$$ 
when $sgn(i)=1$. These factors can be rewritten uniformly as 
$$[z]_{a,b;q,p}-[-sgn(i)b_i]_{a,b;q,p}=
W_{a,b;q,p}(-sgn(i)b_i)[z+sgn(i)b_i]_{a q^{-2 sgn(i)b_i},b q^{-sgn(i)b_i};q,p}.$$
Similarly, we rewrite the factors in the right hand side of
\eqref{eqn:elptMRpf} by using the identity
$$[z]_{a,b;q,p}-[-\overline{A}_j]_{a,b;q,p}=W_{a,b;q,p}(-\overline{A}_j)
[z+\overline{A}_j]_{a q^{-2\overline{A}_j},b q^{-\overline{A}_j};q,p}.$$ 
Then replacing $M\!R_{n-k} ^\A (a,b;q,p; \B^\A, sgn, \overline{sgn})$ by
$R_{n-k}^\A (a,b;q,p; \B^\A, sgn, \overline{sgn})$ which was
defined in \eqref{eqn:Rdef} takes care of the extra factors and in the end
we obtain \eqref{eqn:elptMR}.
\end{proof}

\begin{remark}
In previous work \cite{SY}, we have defined elliptic extensions of
$q$-rook numbers and $q$-file numbers, $r_{k}(a,b;q,p;B)$ and
$f_{k}(a,b;q,p;B)$ respectively, 
and obtained the following product formulas: for any 
Ferrers board $B=B(b_1,\dots, b_n)$,
\begin{equation}\label{eqn:elptrook}
\prod_{i=1}^n [z+b_i -i+1]_{aq^{2(i-1-b_i)},bq^{i-1-b_i};q,p}=
\sum_{k=0}^n r_{n-k}(a,b;q,p;B)
\prod_{j=1}^k[z-j+1]_{aq^{2(j-1)},bq^{j-1};q,p},
\end{equation}
and for any skyline board $B=B(c_1,\dots, c_n)$,
\begin{equation}\label{eqn:elptfile}
\prod_{i=1}^n [z+c_i]_{aq^{-2c_i},bq^{-c_i};q,p}=
\sum_{k=0}^n f_{n-k}(a,b;q,p;B)([z]_{a, b;q,p})^k.
\end{equation}
If we set $sgn(i)=1$ and $\overline{sgn}(i)=-1$ for all $i=1,\dots, n$, 
$\A =(0,1,1,\dots, 1)$ and $\B =(b_1, b_2-1, b_3-2,\dots, b_n-n+1)$,
then the product formula in Theorem \ref{thm:elptMR} becomes
\eqref{eqn:elptrook}.

On the other hand, if we set $sgn(i)=1$ for all $i=1,\dots, n$,
$\A =(0,0, \dots, 0)$
(hence we do not need to define $\overline{sgn}$), $\B =(c_1,c_2,\dots, c_n)$,
then we obtain \eqref{eqn:elptfile} as a result of Theorem \ref{thm:elptMR}.
\end{remark}

%%%%%%%%%%%%%%%%%%%%%%%%%%%%%%%%%%%%%%%%%%%%%%%%%%%%%%%%%%%%%%%%

\section{Realization of other rook models}

%%%%%%%%%%%%%%%%%%%%%%%%%%%%%%%%%%%%%%%%%%%%%%%%%%%%%%%%%%%%%%%5

\subsection{$\J$-attacking rook model}\label{subsec:jattack}

In \cite{RW}, Remmel and Wachs introduced the \emph{$\J$-attacking rook model}. 
Fix an integer $\J\ge 1$. A Ferrers board $B(b_1,\dots, b_n)$ is called a 
\emph{$\J$-attacking board} if $b_{i+1}\ge b_i +\J -1$, for $b_i\ne 0$,
$1\le i <n$. 
Given a $\J$-attacking board $B(b_1,\dots, b_n)$, a rook $\mathbf r\in B$ 
\emph{$\J$-attacks} a cell $c\in B(b_1,\dots, b_n)$ if $c$ lies in a column
which is strictly to the right of the column containing $\mathbf r$ and
$c$ lies in the first $\J$ rows which are weakly above the row of $\mathbf r$
and which are not $\J$-attacked by any rook which lies in a column that is 
strictly to the left of $\mathbf r$. A placement $P$ of $k$ rooks in $B$
is called \emph{$\J$-nonattacking} if each column contains at most one rook
and each rook does not $\J$-attack other rooks. Given a $\J$-nonattacking
rook placement $P$, a rook $\mathbf r\in P$ cancels the cells in the same
column below it and the cells which are 
$\J$-attacked by $\mathbf r$. 

Given a $\J$-attacking board $B$, let $\mathcal N_k ^{~\J} (B)$
denote the set of all $\J$-nonattacking placements of $k$ rooks in $B$.
For any placement $P\in \mathcal N_k ^{~\J} (B)$,
let $u_B ^\J (P)$ denote the number of cells in $B-P$ which are not cancelled 
by any rook in $P$. If we define the $q$-rook number of $B$ by 
$$r_k ^\J(q;B)= \sum_{P\in\mathcal N_k ^{~\J} (B)} q^{u_B ^\J (P)},$$
then we have the following product formula.

\begin{theorem}\cite{RW}\label{thm:RW}
 Given a $\J$-attacking board $B=B(b_1,\dots, b_n)$, 
 \begin{equation}\label{eqn:RW}
 \prod_{i=1}^n [z+b_i -\J(i-1)]_q =
\sum_{k=0}^n r_{n-k} ^\J (q;B)[z]_q \!\downarrow_{k,\J },
\end{equation}
where $[z]_q\!\downarrow_{0,\J}=1$ and for $k>0$,
$[z]_q \!\downarrow_{k,\J}=[z]_q [z-\J]_q\cdots [z-(k-1)\J]_q$.
\end{theorem}

Note that in the case when $\J=1$, if we denote $r_k ^1 (q;B)$ by $r_k (q;B)$,
then we recover the $q$-rook numbers of Garsia and Remmel~\cite{GR}
and the product formula \eqref{eqn:qrook}.
We can obtain an elliptic extension of the product formula \eqref{eqn:RW} 
from Theorem \ref{thm:elptMR}. Let $\A_{\J,n} =(0, \J, \dots, \J)$,
$\B_{\J,n}=(b_1, b_2-\J, \dots, b_n-\J(n-1))$, for $b_i\ge \J(i-1)$,
$sgn(i)=1$ and $\overline{sgn}(i)=-1$, for all $i=1,\dots, n$. 
In this setting, \eqref{eqn:elptMR} becomes 
\begin{align}\label{eqn:elptRW}
&\prod_{i=1}^n [z+b_i-\J(i-1)]_{aq^{2(\J(i-1)-b_i)},bq^{\J(i-1)-b_i};q,p}\notag\\
&=\sum_{k=0}^n R_{n-k} ^\A(a,b;q,p;\B_{\J,n} ^{\A_{\J,n}},sgn, \overline{sgn})
\prod_{i=0}^{k-1}[z-i\J]_{aq^{2i\J},bq^{i\J};q,p}.
\end{align}
If we set $\J=1$ in \eqref{eqn:elptRW}, then we get 
\begin{align}\label{eqn:elptGR}
&\prod_{i=1}^n [z+b_i-i+1]_{aq^{2(i-1-b_i)},bq^{i-1-b_i};q,p}\notag\\
&=\sum_{k=0}^n R_{n-k} ^\A(a,b;q,p;\B_{1,n} ^{\A_{1,n}},sgn, \overline{sgn})
\prod_{i=0}^{k-1}[z-i]_{aq^{2i},bq^{i};q,p}
\end{align}
which can be considered as an elliptic extension of \eqref{eqn:qrook}.
Note that this result is consistent with what we obtained in \cite{SY}.

In \cite{SY} also an elliptic extension of
Haglund and Remmel's~\cite{HR} rook-theoretic model for
perfect matchings was obtained (and even further generalized).
The $\J=2$ case of \eqref{eqn:elptRW}
\begin{align}
&\prod_{i=1}^n [z+b_i-2(i-1)]_{aq^{2(2(i-1)-b_i)},bq^{2(i-1)-b_i};q,p}\notag\\
&=\sum_{k=0}^n R_{n-k} ^\A(a,b;q,p;\B_{2,n} ^{\A_{2,n}},sgn, \overline{sgn})
\prod_{i=0}^{k-1}[z-2i]_{aq^{4i},bq^{2i};q,p}
\end{align}
gives the same product formula as the one
considered in \cite[Sec.~6]{SY}.

\begin{remark}
In \cite{BR0}, Briggs and Remmel defined the $m$-rook numbers and the
$q$-analogue of them, and proved the following product formula:
let $B=B(b_1,\dots, b_n)$ be a Ferrers board satisfying
$0\le b_1\le \cdots \le b_n\le mn$ 
and for each $1\le i\le n-1$, if $b_i=a_i m + b_i$, $1\le b_i < m$, then
$b_{i+1}\ge (a_i +1)m$.
Then 
\begin{equation}\label{eqn:mrook}
\prod_{i=1}^n [mz+b_i-m(i-1)]_q =
\sum_{k=0}^n r_{n-k} ^m (q;B)[mz]_q\!\downarrow_{k,m}.
\end{equation}
We can obtain an elliptic extension of the product formula \eqref{eqn:mrook} 
from \eqref{eqn:elptRW} by setting $\J =m$ and replacing $z$ by $mz$. 
\end{remark}

Remmel and Wachs use the $\J$-attacking rook model to explain 
the generalized Stirling numbers $s_{n,k}^{\I, \J}(\mathfrak p,q)$ and
$S_{n,k}^{\I, \J}(\mathfrak p,q)$
combinatorially in \cite{RW}. The generalized Stirling numbers of the
second kind 
$S_{n,k}^{\I, \J}(\mathfrak p,q)$ can be defined by the recurrence relation
$$
S_{0,0}^{\I, \J}(\mathfrak p,q)=1, \quad S_{n,k}^{\I, \J}(\mathfrak p,q)=0
\text{ if $k<0$ or $k>n$},
$$
$$
S_{n+1,k}^{\I, \J}(\mathfrak p,q)=S_{n,k-1}^{\I, \J}(\mathfrak p,q)+
[\I+k\J]_{\mathfrak p,q}S_{n,k}^{\I, \J}(\mathfrak p,q),
$$
and also they satisfy 
\begin{equation}\label{eqn:str2}
[z]_{\mathfrak p,q} ^n =\sum_{k=0}^n S_{n,k}^{\I, \J}(\mathfrak p,q)
([z]_{\mathfrak p,q}-[\I]_{\mathfrak p,q}) ([z]_{\mathfrak p,q}-[\I+\J]_{\mathfrak p,q})
\cdots ([z]_{\mathfrak p,q}-[\I +(k-1)\J]_{\mathfrak p,q}),
\end{equation}
where $[n]_{\mathfrak p, q}=\frac{\mathfrak p ^n - q^n}{\mathfrak p -q}$.
Note that we used a different font ``$\mathfrak p$''
to distinguish from the nome $p$ in elliptic functions. 

If we set $\A_{\I,\J,n} =(\I, \J, \dots, \J)$, $\B_0 =(0, \dots, 0)$, 
$\overline{sgn}(i)=-1$, for all $i=1,\dots, n$, \eqref{eqn:elptMRpf} becomes 
\begin{equation}\label{eqn:elptstr2}
[z]_{a,b;q,p} ^n =\sum_{k=0}^n
M\!R_{n-k}^\A(a,b;q,p;\B_0 ^{\A_{\I, \J,n}},sgn,\overline{sgn}
\prod_{s=1}^k\left([z]_{a,b;q,p}-[\I +(s-1)\J]_{a,b;q,p} \right) 
\end{equation}
which can be considered as an elliptic extension of \eqref{eqn:str2}.
Thus we define 
$$
S_{n,k}^{\I, \J}(a,b;q,p):= 
M\!R_{n-k}^\A(a,b;q,p;\B_0 ^{\A_{\I, \J,n}},sgn,\overline{sgn})
$$
as an elliptic extension of $S_{n,k}^{\I, \J}(\mathfrak p,q)$.
From \eqref{eqn:str2}, we get 
the recurrence relation of $S_{n,k}^{\I, \J}(a,b;q,p)$
$$S_{n+1,k}^{\I, \J}(a,b;q,p)=S_{n,k-1}^{\I, \J}(a,b;q,p)+[\I +k\J]_{a,b;q,p}
S_{n,k}^{\I, \J}(a,b;q,p)$$
which may be used to define $S_{n,k}^{\I, \J}(a,b;q,p)$ uniquely with the
conditions 
$$S_{0,0}^{\I, \J}(a,b;q,p)=1 ~\text{ and }~ S_{n,k}^{\I, \J}(a,b;q,p)=0
\text{ for $k<0$ or $k>n$}.$$

On the other hand, the generalized Stirling numbers of the first kind 
$s_{n,k}^{\I, \J}(\mathfrak p,q)$ are defined by 
$$s_{0,0}^{\I, \J}(\mathfrak p,q)=1, \quad s_{n,k}^{\I, \J}(\mathfrak p,q)=0
\text{ if $k<0$ or $k>n$},$$
$$s_{n+1,k}^{\I, \J}(\mathfrak p,q)=s_{n,k-1}^{\I, \J}(\mathfrak p,q)
-[\I+n\J]_{\mathfrak p,q}s_{n,k}^{\I, \J}(\mathfrak p,q),$$
and they have the generating function 
\begin{equation}\label{eqn:str1}
([z]_{\mathfrak p,q}-[\I]_{\mathfrak p,q})([z]_{\mathfrak p,q}-[\I+\J]_{\mathfrak p,q})
\cdots ([z]_{\mathfrak p,q}-[\I+(n-1)\J]_{\mathfrak p,q})
=\sum_{k=0}^n s_{n,k}^{\I, \J}(\mathfrak p,q)[z]_{\mathfrak p,q} ^k.
\end{equation}
If we set $\A_0 =(0,0,\dots, 0)$ and
$\B_{\I,\J,n}=(\I, \I+\J, \dots, \I+(n-1)\J)$, $sgn(i)=-1$
for all $i=1,\dots, n$, \eqref{eqn:elptMRpf} becomes 
\begin{align}\label{eqn:elptstr1}
([z]_{a,b;q,p}-[\I]_{a,b;q,p})([z]_{a,b;q,p}-[\I+\J]_{a,b;q,p})\cdots
([z]_{a,b;q,p}-[\I+(n-1)\J]_{a,b;q,p})\notag&\\
=\sum_{k=0}^n  M\!R_{n-k}^\A(a,b;q,p;\B_{\I,\J,n} ^{\A_{0}},sgn,\overline{sgn})
[z]_{a,b;q,p} ^k&.
\end{align}
We let $s_{n,k}^{\I, \J}(a,b;q,p):=
M\!R_{n-k}^\A(a,b;q,p;\B_{\I,\J,n} ^{\A_{0}},sgn,\overline{sgn})$ 
which defines an elliptic extension of $s_{n,k}^{\I, \J}(\mathfrak p,q)$.
Then \eqref{eqn:elptstr1} 
gives the recurrence relation 
$$s_{n+1,k}^{\I, \J}(a,b;q,p)=s_{n,k-1}^{\I, \J}(a,b;q,p)-
[\I+n\J]_{a,b;q,p}s_{n,k}^{\I, \J}(a,b;q,p)$$
which can be used to define $s_{n,k}^{\I, \J}(a,b;q,p)$ uniquely with the
conditions 
$s_{0,0}^{\I, \J}(a,b;q,p)=1$ and $s_{n,k}^{\I, \J}(a,b;q,p)=0$ if $k<0$ or $k>n$. 

As the matrices $||s_{n,k}^{\I, \J}(\mathfrak p,q)||$ and
$||S_{n,k}^{\I, \J}(\mathfrak p,q)||$
are inverses of each other, the elliptic extensions
$||s_{n,k}^{\I, \J}(a,b;q,p)||$ and 
$||S_{n,k}^{\I, \J}(a,b;q,p)||$ share the same property.
The idea used to prove this property 
in \cite{RW} works in this case similarly. 

\begin{proposition}
For all $0\le r \le n$, 
\begin{equation}
\sum_{k=r}^n  S_{n,k}^{\I, \J}(a,b;q,p)s_{k,r}^{\I, \J}(a,b;q,p) =\chi(r=n),
\end{equation}
where $\chi(A)=1$ if the statement $A$ is true and $0$ otherwise.
\end{proposition}
\begin{proof}
Recall the definitions 
\begin{align*}
S_{n,k}^{\I, \J}(a,b;q,p) &= M\!R_{n-k}^\A (a,b;q,p;\B_0 ^{\A_{\I,\J,n}},sgn,
\overline{sgn})\\
&= \sum_{P\in \mathcal N_{n-k} ^\A (\B _0 ^{\A_{\I, \J,n}})}
M_{a,b;q,p}(\B_0 ^{\A_{\I, \J,n}},sgn,\overline{sgn};P),
\end{align*}
and 
\begin{align*}
s_{k,r}^{\I, \J}(a,b;q,p) &= M\!R_{k-r}^\A (a,b;q,p;\B_{\I, \J,k} ^{\A_{0}},sgn,
\overline{sgn})\\
&= \sum_{P\in \mathcal N_{k-r} ^\A (\B _{\I, \J,k} ^{\A_{0}})}
M_{a,b;q,p}(\B_{\I, \J,k} ^{\A_{0}},sgn,\overline{sgn};P),
\end{align*}
where $M_{a,b;q,p}(\B ^{\A},sgn,\overline{sgn};P)$ is defined in \eqref{eqn:Mdef},
hence we want to show that 
\begin{align*}
&{} \sum_{k=r}^n  S_{n,k}^{\I, \J}(a,b;q,p)s_{k,r}^{\I, \J}(a,b;q,p)\\
&= \sum_{k=r}^n \left(\sum_{(P,Q)\in \mathcal N_{n-k} ^\A (\B _0 ^{\A_{\I, \J,n}})
\times \mathcal N_{k-r} ^\A (\B _{\I, \J,k} ^{\A_{0}})} 
M_{a,b;q,p}(\B_0 ^{\A_{\I, \J,n}},sgn,\overline{sgn};P)
M_{a,b;q,p}(\B_{\I, \J,k} ^{\A_{0}},sgn,\overline{sgn};Q)
\right)\\
&= \begin{cases}
    1 & \text{ if } r=n,\\
    0 & \text{ otherwise .}
   \end{cases}
\end{align*}

For $r=n$, 
\begin{align*}
 &{} S_{n,n}^{\I, \J}(a,b;q,p) s_{n,n}^{\I, \J}(a,b;q,p)\\
 &= 
 M\!R_{0}^\A (a,b;q,p;\B_0 ^{\A_{\I,\J,n}},sgn, \overline{sgn})
M\!R_{0}^\A (a,b;q,p;\B_{\I, \J,n} ^{\A_{0}},sgn, \overline{sgn})=1.
\end{align*}
Now suppose that $r<n$. Consider the elements 
$$(P,Q)\in \bigcup_{k=r}^n \mathcal N_{n-k} ^\A (\B _0 ^{\A_{\I, \J,n}})\times 
\mathcal N_{k-r} ^\A (\B _{\I, \J,k} ^{\A_{0}})$$
and partition them into three classes:
\begin{itemize}
 \item[(i)] $(P,Q)\in \text{Class I }$ if $P$ has a rook in the last column
of $\B_0 ^{\A_{\I,\J,n}}$,
 \item[(ii)] $(P,Q)\in \text{Class II }$ if $P$ has no rook in the
last column of $\B_0 ^{\A_{\I,\J,n}}$,
 but there is a rook of $Q$ in the last column of $\B_{\I, \J,k} ^{\A_{0}}$,
 \item[(iii)] $(P,Q)\in \text{Class III }$ if $P$ has no rook in the last
column of $\B_0 ^{\A_{\I,\J,n}}$
 and $Q$ has no rook in the last column of $\B_{\I, \J,k} ^{\A_{0}}$.
\end{itemize}
We make a correspondence between the elements in Class I and Class II such
that the sum of the weights becomes zero. 
More precisely, given $(P,Q)\in  \mathcal N_{n-k} ^\A (\B _0 ^{\A_{\I, \J,n}})
\times  \mathcal N_{k-r} ^\A (\B _{\I, \J,k} ^{\A_{0}})$ of Class I, we find
$(P', Q')\in  \mathcal N_{n-k-1} ^\A (\B _0 ^{\A_{\I, \J,n}})\times 
\mathcal N_{k-r+1} ^\A (\B _{\I, \J,k+1} ^{\A_{0}})$
of Class II such that 
\begin{align*}
&{}M_{a,b;q,p}(\B_0 ^{\A_{\I, \J,n}},sgn,\overline{sgn};P)
M_{a,b;q,p}(\B_{\I, \J,k} ^{\A_{0}},sgn,\overline{sgn};Q)\\
&+M_{a,b;q,p}(\B_0 ^{\A_{\I, \J,n}},sgn,\overline{sgn};P')
M_{a,b;q,p}(\B_{\I, \J,k+1} ^{\A_{0}},sgn,\overline{sgn};Q')=0.
\end{align*}
Since $(P,Q)$ is in Class I, $P$ has a rook in the last column of
$\B_0 ^{\A_{\I,\J,n}}$, and the $n-k-1$ rooks to the left of the last column
cancel $(n-k-1)\J$ cells in the last column corresponding 
to the $a_n$, $a_{n-1},\dots, a_{k+2}$-parts, and there are
$\I +(n-1)\J -(n-k-1)\J=\I+k\J$ uncancelled cells 
which have assigned weights $1, W_{a,b;q,p}(1),\dots, W_{a,b;q,p}(\I+k\J-1)$
from bottom to top. Then define $P'$ to be the placement $P$ after
removing the rook in the last column of $\B_0 ^{\A_{\I,\J,n}}$
and $Q'$ to be the result of attaching an extra column of height $\I+k\J$
to the right of the placement $Q$ such that this extra column contains
a rook in the $t$-th cell from the bottom, if $P$ had the last rook
in the $t$-th cell from the bottom. Note that by the way of assigning 
weights to the cells, this last column in $Q'$ has weights
$-1, -W_{a,b;q,p}(1),\dots,- W_{a,b;q,p}(\I+k\J-1)$ from the bottom. Hence,
the correspondence $(P,Q)\to (P',Q')$ would preserve the product of weights 
but changes the signs. Reversing this correspondence can be easily described.
If $(P',Q')$ is in Class II, then $Q$ is the result of removing the
last column of $Q'$ and $P$ is the result of putting a rook in the
last column of $\B_0 ^{\A_{\I,\J,n}}$ in the $t$-th row from the bottom,
if $Q'$ had a rook in the $t$-th row from the bottom. See
Figure~\ref{fig:map1} for an example of this map.
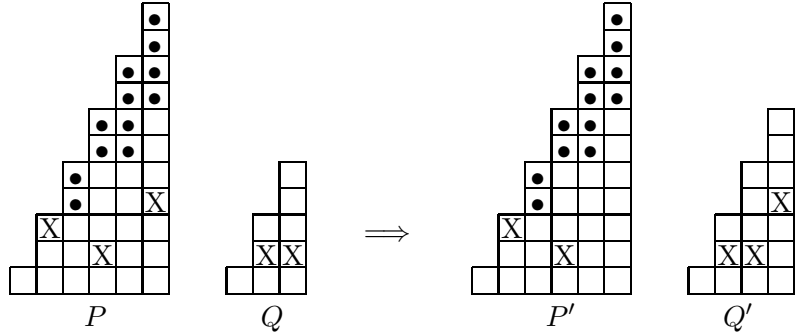
\begin{figure}[ht]
$$\begin{picture}(60,120)(0,0)
\multiput(0,10)(0,10){2}{\line(1,0){60}}
\multiput(10,30)(0,10){2}{\line(1,0){50}}
\multiput(20,50)(0,10){2}{\line(1,0){40}}
\multiput(30,70)(0,10){2}{\line(1,0){30}}
\multiput(40,90)(0,10){2}{\line(1,0){20}}
\multiput(50,110)(0,10){2}{\line(1,0){10}}
\multiput(0,10)(10,0){1}{\line(0,1){10}}
\multiput(10,10)(10,0){1}{\line(0,1){30}}
\multiput(20,10)(10,0){1}{\line(0,1){50}}
\multiput(30,10)(10,0){1}{\line(0,1){70}}
\multiput(40,10)(10,0){1}{\line(0,1){90}}
\multiput(50,10)(10,0){2}{\line(0,1){110}}
\put(28, -2){$P$}
\put(11,31){X}
\put(31,21){X}
\put(51, 41){X}
\put(22,41){$\bullet$}
\put(22,51){$\bullet$}
\put(32,61){$\bullet$}
\put(32,71){$\bullet$}
\put(42,61){$\bullet$}
\put(42,71){$\bullet$}
\put(42,81){$\bullet$}
\put(42,91){$\bullet$}
\put(52,111){$\bullet$}
\put(52,101){$\bullet$}
\put(52,81){$\bullet$}
\put(52,91){$\bullet$}
\end{picture}
\qquad 
\begin{picture}(30,50)(0,0)
\multiput(0,10)(0,10){2}{\line(1,0){30}}
\multiput(10,30)(0,10){2}{\line(1,0){20}}
\multiput(20,50)(0,10){2}{\line(1,0){10}}
\multiput(0,10)(10,0){1}{\line(0,1){10}}
\multiput(10,10)(10,0){1}{\line(0,1){30}}
\multiput(20,10)(10,0){2}{\line(0,1){50}}
\put(13, -2){$Q$}
\put(11,21){X}
\put(21,21){X}
\end{picture}
\qquad 
\begin{picture}(30,30)(0,0)
\put(0,30){$\Longrightarrow$} 
\end{picture}
\quad 
\begin{picture}(60,120)(0,0)
\multiput(0,10)(0,10){2}{\line(1,0){60}}
\multiput(10,30)(0,10){2}{\line(1,0){50}}
\multiput(20,50)(0,10){2}{\line(1,0){40}}
\multiput(30,70)(0,10){2}{\line(1,0){30}}
\multiput(40,90)(0,10){2}{\line(1,0){20}}
\multiput(50,110)(0,10){2}{\line(1,0){10}}
\multiput(0,10)(10,0){1}{\line(0,1){10}}
\multiput(10,10)(10,0){1}{\line(0,1){30}}
\multiput(20,10)(10,0){1}{\line(0,1){50}}
\multiput(30,10)(10,0){1}{\line(0,1){70}}
\multiput(40,10)(10,0){1}{\line(0,1){90}}
\multiput(50,10)(10,0){2}{\line(0,1){110}}
\put(28, -2){$P'$}
\put(11,31){X}
\put(31,21){X}
\put(22,41){$\bullet$}
\put(22,51){$\bullet$}
\put(32,61){$\bullet$}
\put(32,71){$\bullet$}
\put(42,61){$\bullet$}
\put(42,71){$\bullet$}
\put(42,81){$\bullet$}
\put(42,91){$\bullet$}
\put(52,111){$\bullet$}
\put(52,101){$\bullet$}
\put(52,81){$\bullet$}
\put(52,91){$\bullet$}
\end{picture}
\qquad 
\begin{picture}(30,50)(0,0)
\multiput(0,10)(0,10){2}{\line(1,0){40}}
\multiput(10,30)(0,10){2}{\line(1,0){30}}
\multiput(20,50)(0,10){2}{\line(1,0){20}}
\multiput(30,70)(0,10){2}{\line(1,0){10}}
\multiput(0,10)(10,0){1}{\line(0,1){10}}
\multiput(10,10)(10,0){1}{\line(0,1){30}}
\multiput(20,10)(10,0){1}{\line(0,1){50}}
\multiput(30,10)(10,0){2}{\line(0,1){70}}
\put(13, -2){$Q'$}
\put(11,21){X}
\put(21,21){X}
\put(31, 41){X}
\end{picture}
$$
\caption{An example of correspondence from Class I to Class II}\label{fig:map1}
\end{figure} 
This bijection implies that we only need to consider the elements in Class III. 

Note that if $r=0$, then there are no elements in Class III since every element 
$(P,Q)\in \mathcal N_{n-k} ^\A (\B _0 ^{\A_{\I, \J,n}})\times 
\mathcal N_{k-0} ^\A (\B _{\I, \J,k} ^{\A_{0}})$
has a rook of $Q$ in the last column of $\B _{\I, \J,k} ^{\A_{0}}$. Hence, 
in the case when $r=0$, the correspondence between Class I and Class II proves 
$$\sum_{k=0}^n  S_{n,k}^{\I, \J}(a,b;q,p)s_{k,0}^{\I, \J}(a,b;q,p)=0.$$

For $r\ge 1$, there is a weight-preserving bijection between Class III and
the set
\hfill\linebreak
$\bigcup_{k=r-1}^{n-1}\mathcal N_{n-1-k} ^\A (\B _0 ^{\A_{\I, \J,n-1}})\times
\mathcal N_{k-(r-1)} ^\A (\B _{\I, \J,k} ^{\A_{0}})$,
simply by removing (and adding, in the other direction of the correspondence)
the empty last columns. 
See Figure~\ref{fig:map2} for an example.
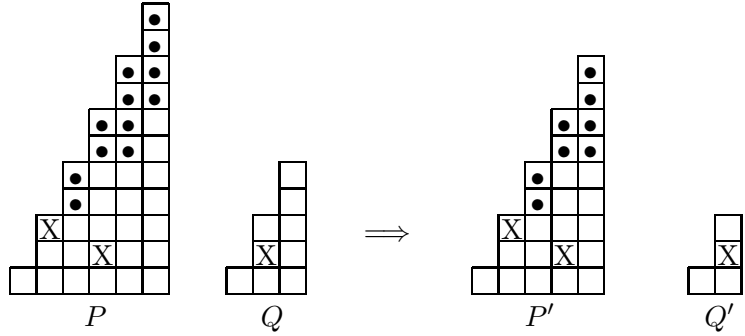
\begin{figure}[ht]
$$\begin{picture}(60,120)(0,0)
\multiput(0,10)(0,10){2}{\line(1,0){60}}
\multiput(10,30)(0,10){2}{\line(1,0){50}}
\multiput(20,50)(0,10){2}{\line(1,0){40}}
\multiput(30,70)(0,10){2}{\line(1,0){30}}
\multiput(40,90)(0,10){2}{\line(1,0){20}}
\multiput(50,110)(0,10){2}{\line(1,0){10}}
\multiput(0,10)(10,0){1}{\line(0,1){10}}
\multiput(10,10)(10,0){1}{\line(0,1){30}}
\multiput(20,10)(10,0){1}{\line(0,1){50}}
\multiput(30,10)(10,0){1}{\line(0,1){70}}
\multiput(40,10)(10,0){1}{\line(0,1){90}}
\multiput(50,10)(10,0){2}{\line(0,1){110}}
\put(28, -2){$P$}
\put(11,31){X}
\put(31,21){X}
%\put(51, 41){X}
\put(22,41){$\bullet$}
\put(22,51){$\bullet$}
\put(32,61){$\bullet$}
\put(32,71){$\bullet$}
\put(42,61){$\bullet$}
\put(42,71){$\bullet$}
\put(42,81){$\bullet$}
\put(42,91){$\bullet$}
\put(52,111){$\bullet$}
\put(52,101){$\bullet$}
\put(52,81){$\bullet$}
\put(52,91){$\bullet$}
\end{picture}
\qquad 
\begin{picture}(30,50)(0,0)
\multiput(0,10)(0,10){2}{\line(1,0){30}}
\multiput(10,30)(0,10){2}{\line(1,0){20}}
\multiput(20,50)(0,10){2}{\line(1,0){10}}
\multiput(0,10)(10,0){1}{\line(0,1){10}}
\multiput(10,10)(10,0){1}{\line(0,1){30}}
\multiput(20,10)(10,0){2}{\line(0,1){50}}
\put(13, -2){$Q$}
\put(11,21){X}
%\put(21,21){X}
\end{picture}
\qquad 
\begin{picture}(30,30)(0,0)
\put(0,30){$\Longrightarrow$} 
\end{picture}
\quad 
\begin{picture}(60,120)(0,0)
\multiput(0,10)(0,10){2}{\line(1,0){50}}
\multiput(10,30)(0,10){2}{\line(1,0){40}}
\multiput(20,50)(0,10){2}{\line(1,0){30}}
\multiput(30,70)(0,10){2}{\line(1,0){20}}
\multiput(40,90)(0,10){2}{\line(1,0){10}}
%\multiput(50,110)(0,10){2}{\line(1,0){10}}
\multiput(0,10)(10,0){1}{\line(0,1){10}}
\multiput(10,10)(10,0){1}{\line(0,1){30}}
\multiput(20,10)(10,0){1}{\line(0,1){50}}
\multiput(30,10)(10,0){1}{\line(0,1){70}}
\multiput(40,10)(10,0){2}{\line(0,1){90}}
%\multiput(50,10)(10,0){2}{\line(0,1){110}}
\put(20, -2){$P'$}
\put(11,31){X}
\put(31,21){X}
\put(22,41){$\bullet$}
\put(22,51){$\bullet$}
\put(32,61){$\bullet$}
\put(32,71){$\bullet$}
\put(42,61){$\bullet$}
\put(42,71){$\bullet$}
\put(42,81){$\bullet$}
\put(42,91){$\bullet$}
%\put(52,111){$\bullet$}
%\put(52,101){$\bullet$}
%\put(52,81){$\bullet$}
%\put(52,91){$\bullet$}
\end{picture}
\qquad 
\begin{picture}(30,50)(0,0)
\multiput(0,10)(0,10){2}{\line(1,0){20}}
\multiput(10,30)(0,10){2}{\line(1,0){10}}
%\multiput(20,50)(0,10){2}{\line(1,0){20}}
%\multiput(30,70)(0,10){2}{\line(1,0){10}}
\multiput(0,10)(10,0){1}{\line(0,1){10}}
\multiput(10,10)(10,0){2}{\line(0,1){30}}
%\multiput(20,10)(10,0){2}{\line(0,1){50}}
%\multiput(30,10)(10,0){2}{\line(0,1){70}}
\put(6, -2){$Q'$}
\put(11,21){X}
%\put(21,21){X}
%\put(31, 41){X}
\end{picture}
$$
\caption{An example of a map from Class III to 
$\bigcup_{k=r-1}^{n-1}\mathcal N_{n-1-k} ^\A (\B _0 ^{\A_{\I, \J,n-1}})\times
\mathcal N_{k-(r-1)} ^\A (\B _{\I, \J,k} ^{\A_{0}})$.}\label{fig:map2}
\end{figure}
Thus the two bijections 
that we constructed above explains 
\begin{align*}
&{}\sum_{k=r}^n  S_{n,k}^{\I, \J}(a,b;q,p)s_{k,r}^{\I, \J}(a,b;q,p)\\
={}
&{}\sum_{k=r-1}^{n-1}  S_{n-1,k}^{\I, \J}(a,b;q,p)s_{k,r-1}^{\I, \J}(a,b;q,p)\\
={}&{}\chi(n-1=r-1)
\end{align*}
where the last equality comes from the induction hypothesis. 
\end{proof}

%%%%%%%%%%%%%%%%%%%%%%%%%%%%%%%%%%%%%%%%%%%%%%%%%%%%%%%%%%%%%%%%%%%%%%%

\subsection{Elliptic extension of the Stirling numbers}

%%%%%%%%%%%%%%%%%%%%%%%%%%%%%%%%%%%%%%%%%%%%%%%%%%%%%%%%%%%%%%%%%%%%%

In the generalized Stirling numbers $S_{n,k}^{\I, \J}(a,b;q,p)$ and
$s_{n,k}^{\I, \J}(a,b;q,p)$ considered in Section \ref{subsec:jattack},
if we set $\I=0$ and $\J=1$, we can consider them as elliptic extensions
of the $q$-Stirling numbers $S_q (n,k)$ and $s_q(n,k)$. 
To be consistent with the notation for $q$-Stirling numbers, let us denote 
$$S_{a,b;q,p}(n,k):= S_{n,k}^{0, 1}(a,b;q,p),\qquad s_{a,b;q,p}(n,k):=
s_{n,k}^{0, 1}(a,b;q,p).$$
Note that if we let $p\to 0$, $a\to 0$ and $b\to 0$ in this order 
(or $p\to 0$, $b\to 0$ and $a\to \infty$ in this order), then the
elliptic extensions converge to the $q$-analogues of the Stirling numbers
$S_q (n,k)$ and $s_q(n,k)$.
In \cite{ML}, de M\'{e}dicis and Leroux introduced and studied
$\mathfrak A$-Stirling numbers which are generalizations of the
Stirling numbers of the second and first kind.
By setting $w_i =[i]_{a,b;q,p}$ in their setting (see \cite{ML} for details)
we can obtain the elliptic extensions $S_{a,b;q,p}(n,k)$ and $s_{a,b;q,p}(n,k)$.
In \cite{ML}, the generalizations of convolution formulae has been studied
and proved by using the $\mathcal A$-tableaux which are the generalizations of 
the $0$-$1$-tableaux. Here, we state the convolution formulae of
$S_{a,b;q,p}(n,k)$ 
and $s_{a,b;q,p}(n,k)$, and prove them by using the augmented rook board. Let 
$c_{a,b;q,p}(n,k)=(-1)^{n-k}s_{a,b;q,p}(n,k)$ denote the unsigned Stirling numbers 
of the first kind.

\begin{proposition}
The elliptic extension of the Stirling numbers of the second kind
$S_{a,b;q,p}(n,k)$
satisfy 
\begin{multline}\label{eqn:conStir2-1}
S_{a,b;q,p}(m+n, k)=\sum_{i=0}^k \sum_{j=k-i}^m \binom{m}{j}([i]_{a,b;q,p}) ^{m-j}
(W_{a,b;q,p}(i))^{i+j-k}\\
\times S_{a,b;q,p}(n,i)S_{a q^{2i}, b q^i;q,p}(j, k-i),
\end{multline}
\begin{multline}\label{eqn:conStir2-2}
S_{a,b;q,p}(n+1,k+l+1)=
\sum_{i=0}^n \sum_{j=0}^{n-l-i}\left\{ 
\binom{n-i}{j}(W_{a,b;q,p}(k+1))^{n-l-i-j}\right.\\
\left.\times [k+1]_{a,b;q,p}^j S_{a,b;q,p}(i,k)S_{a q^{2k+2},b q^{k+1};q,p}(n-i-j,l)
\right\}, 
\end{multline}
and the elliptic extension of the signless Stirling numbers of the
first kind $c_{a,b;q,p}(n,k)$ satisfy
\begin{multline}\label{eqn:conStir1-1}
c_{a,b;q,p}(m+n, k)=\sum_{i=0}^k \sum_{j=k-i}^m \binom{j}{k-i}
([n]_{a,b;q,p})^{j-k+i}(W_{a,b;q,p}(n))^{m-j}\\
\times c_{a,b;q,p}(n,i)c_{a q^{2n}, b q^n;q,p}(m,j),
\end{multline}
\begin{multline}\label{eqn:conStir1-2}
 c_{a,b;q,p}(n+1,k+l+1)=
\sum_{i=0}^n \sum_{j=0}^{n-l-i}\left\{ 
\binom{j+l}{j}(W_{a,b;q,p}(i+1))^{n-l-i-j}\right.\\
\left.\times ([i+1]_{a,b;q,p})^j c_{a,b;q,p}(i,k)c_{a q^{2i+2},b q^{i+1};q,p}(n-i,j+l)
\right\}. 
\end{multline}
\end{proposition}

\begin{proof}
Recall that the elliptic extension of the Stirling numbers of the second kind
can be realized if we use the sequences $\mathcal A_{0,1,n}=(0,1,1,\dots,1)$,
$\mathcal B_0 =(0,0,\dots, 0)$ and $\overline{sgn}(i)=-1$
for all $i=1,\dots, n$ (and $sgn(i)$ does not matter since $b_i$'s are
all zero). Then 
$$S_{a,b;q,p}(n,k)=M\!R_{n-k}^{\mathcal A}
(a,b;q,p;\mathcal B_0 ^{\mathcal A_{0,1,n}};sgn,\overline{sgn})$$
where $M\!R_{n-k}^{\mathcal A} (a,b;q,p;\mathcal B ^{\mathcal A};sgn,\overline{sgn})$
is defined in \eqref{eqn:MRdef}. The weights of the cells in the
$i$-th column are $1, W_{a,b;q,p}(1),\dots, W_{a,b;q,p}(i-2)$ 
from the bottom, for each $i$. 

We prove \eqref{eqn:conStir2-1} by considering the augmented board
$\mathcal B_0 ^{\mathcal A_{0,1,m+n}}$
where $B_0$ has $m+n$ zeros. The first $n$ columns can be considered as 
$\mathcal B_0 ^{\mathcal A_{0,1,n}}$ and let us place $(n-i)$ rooks in there. 
This takes care of the factor $S_{a,b;q,p}(n,i)$. Those $(n-i)$ rooks cancel 
$(n-i)$ cells from the top in each column to the right of this board, 
and the $(n+1)$-st column has $i$ uncancelled cells.
Now we divide the rest of $\mathcal B_0 ^{\mathcal A_{0,1,m+n}}$
in two parts: a rectangular shape part with height $i$ in the bottom,
and the rest of the upper part.
Let us place $(m-j)$ rooks in the bottom rectangular part.
There are $\binom{m}{m-j}=\binom{m}{j}$ possibilities 
for choosing columns to place those $(m-j)$ rooks, and these possible
placements of $(m-j)$ rooks give the weight contribution
$([i]_{a,b;q,p}) ^{m-j}$. After considering the cancellation of the $(m-j)$
rooks, the rest of the upper part would be considered as a board 
$\mathcal B_0 ^{\mathcal A_{0,1,j}}$. We place $(j-k+i)$ rooks in this board. 
Note that the weights of the cells in this board start from $W_{a,b;q,p}(i)$
in each column and that is why we need the extra factor
$(W_{a,b;q,p}(i))^{i+j-k}$ for each rook 
and the shifts for $a$ and $b$ in $S_{a q^{2i}, b q^i;q,p}(j, k-i)$. 
See Figure~\ref{fig:Str2conv} for an example.
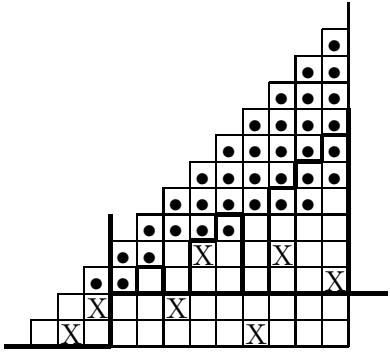
\begin{figure}[ht]
$$\begin{picture}(150,130)(0,10)
\multiput(0,10)(0,10){1}{\line(1,0){130}}
\multiput(10,20)(0,10){1}{\line(1,0){120}}
\multiput(20,30)(0,10){1}{\line(1,0){110}}
\multiput(30,40)(0,10){1}{\line(1,0){100}}
\multiput(40,50)(0,10){1}{\line(1,0){90}}
\multiput(50,60)(0,10){1}{\line(1,0){80}}
\multiput(60,70)(0,10){1}{\line(1,0){70}}
\multiput(70,80)(0,10){1}{\line(1,0){60}}
\multiput(80,90)(0,10){1}{\line(1,0){50}}
\multiput(90,100)(0,10){1}{\line(1,0){40}}
\multiput(100,110)(0,10){1}{\line(1,0){30}}
\multiput(110,120)(0,10){1}{\line(1,0){20}}
\multiput(120,130)(0,10){1}{\line(1,0){10}}
\multiput(10,10)(0,10){1}{\line(0,1){10}}
\multiput(10,10)(0,10){1}{\line(0,1){10}}
\multiput(20,10)(0,10){1}{\line(0,1){20}}
\multiput(30,10)(0,10){1}{\line(0,1){30}}
\multiput(40,10)(0,10){1}{\line(0,1){40}}
\multiput(50,10)(0,10){1}{\line(0,1){50}}
\multiput(60,10)(0,10){1}{\line(0,1){60}}
\multiput(70,10)(0,10){1}{\line(0,1){70}}
\multiput(80,10)(0,10){1}{\line(0,1){80}}
\multiput(90,10)(0,10){1}{\line(0,1){90}}
\multiput(100,10)(0,10){1}{\line(0,1){100}}
\multiput(110,10)(0,10){1}{\line(0,1){110}}
\multiput(120,10)(0,10){1}{\line(0,1){120}}
\multiput(130,10)(0,10){1}{\line(0,1){130}}
%%%%%%%%%%%%%%%%%%%%%%%%%%%%%%%%%%%%%%%%%%%5
\put(21,11){X}
\put(31,21){X}
\put(61,21){X}
\put(91,11){X}
\put(71,41){X}
\put(101,41){X}
\put(121,31){X}
\put(32,31){$\bullet$}
\put(42,41){$\bullet$}
\put(52,51){$\bullet$}
\put(62,61){$\bullet$}
\put(72,71){$\bullet$}
\put(82,81){$\bullet$}
\put(92,91){$\bullet$}
\put(102,101){$\bullet$}
\put(112,111){$\bullet$}
\put(122,121){$\bullet$}
\put(42,31){$\bullet$}
\put(52,41){$\bullet$}
\put(62,51){$\bullet$}
\put(72,61){$\bullet$}
\put(82,71){$\bullet$}
\put(92,81){$\bullet$}
\put(102,91){$\bullet$}
\put(112,101){$\bullet$}
\put(122,111){$\bullet$}
\put(72,51){$\bullet$}
\put(82,61){$\bullet$}
\put(92,71){$\bullet$}
\put(102,81){$\bullet$}
\put(112,91){$\bullet$}
\put(122,101){$\bullet$}
\put(102,71){$\bullet$}
\put(112,81){$\bullet$}
\put(122,91){$\bullet$}
\put(82,51){$\bullet$}
\put(92,61){$\bullet$}
\put(102,61){$\bullet$}
\put(112,71){$\bullet$}
\put(122,81){$\bullet$}
\put(112,61){$\bullet$}
\put(122,71){$\bullet$}
%\put(3,8){$\underbrace{~\quad\qquad}$}
%\put(18, -5){$n$}
%\put(130, 15){$\left. \right\}$}
\thicklines\linethickness{1.3pt}
\multiput(40,10)(0,10){1}{\line(0,1){50}}
\multiput(0,10)(0,10){1}{\line(1,0){40}}
\multiput(40,30)(0,10){1}{\line(1,0){105}}
\multiput(50,30)(10,0){2}{\line(0,1){10}}
\multiput(50,40)(10,0){1}{\line(1,0){10}}
\multiput(70,30)(10,0){1}{\line(0,1){20}}
\multiput(80,50)(10,0){1}{\line(0,1){10}}
\multiput(90,30)(10,0){1}{\line(0,1){30}}
\multiput(70,50)(10,0){1}{\line(1,0){10}}
\multiput(80,60)(10,0){1}{\line(1,0){10}}
\multiput(100,30)(10,0){1}{\line(0,1){40}}
\multiput(130,30)(10,0){1}{\line(0,1){70}}
\multiput(110,70)(10,0){1}{\line(0,1){10}}
\multiput(120,80)(10,0){1}{\line(0,1){10}}
\multiput(100,70)(10,0){1}{\line(1,0){10}}
\multiput(110,80)(10,0){1}{\line(1,0){10}}
\multiput(120,90)(10,0){1}{\line(1,0){10}}
\end{picture}$$
\caption{An example of partitioning the board $\mathcal B_0 ^{\mathcal A_{0,1,m+n}}$
for $m=9$, $n=4$, $i=2$, $j=7$ and $k=3$.}\label{fig:Str2conv}
\end{figure}
The convolution in
\eqref{eqn:conStir2-2} can be explained similarly and so we do not
provide details.

For the elliptic extension of the Stirling numbers of the first kind
$s_{a,b;q,p}(n,k)$, we considered the augmented board
$\mathcal B_{0,1,n}^{\mathcal A_0}$, where $B_{0,1,n}=(0,1,2,\dots, n-1)$,
$\mathcal A_0=(0,0,\dots, 0)$ and the sign function $sgn(i)=-1$ for all
$i=1,\dots, n$. The weights of the cells in the $i$-th column 
are $-1, -W_{a,b;q,p}(1),\dots, -W_{a,b;q,p}(i-2)$ from the bottom, 
and set 
$$s_{a,b;q,p}(n,k)=M\!R_{n-k}^{\mathcal A} (a,b;q,p;
\mathcal B_{0,1,n}^{\mathcal A_0};sgn,\overline{sgn}).$$
Since $c_{a,b;q,p}(n,k)=(-1)^{n-k}s_{a,b;q,p}(n,k)$, we can just think that
the weight of the cells in the $i$-th column are
$1, W_{a,b;q,p}(1),\dots, W_{a,b;q,p}(i-2)$ from the bottom 
in the same augmented board and set 
$$c_{a,b;q,p}(n,k)=M\!R_{n-k}^{\mathcal A} (a,b;q,p;
\mathcal B_{0,1,n}^{\mathcal A_0};sgn,\overline{sgn}).$$
Note that since there are no cells in the upper augmented part,
namely $a_i=0$ for all $i$, the rooks placed in
this board do not cancel any other cells. 

Now we prove \eqref{eqn:conStir1-1}.
We consider $\mathcal B_{0,1,m+n}^{\mathcal A_0}$ and
partition the augmented board in three parts: 
the staircase board of the first $n$-columns $\mathcal B_{0,1,n}^{\mathcal A_0}$,
the rectangular shape board of size $[n]\times [m]$ to the right and bottom
of this first part, and the staircase board $\mathcal B_{0,1,m}^{\mathcal A_0}$
which sits on top of the rectangular shape board. We place $(n-i)$ rooks
in the first part and this procedure contributes $c_{a,b;q,p}(n,i)$ factor.
Then we move to the third part and place $(m-j)$ rooks in
$\mathcal B_{0,1,m}^{\mathcal A_0}$. Since the weight of the bottom cell is
$W_{a,b;q,p}(n)$ in each column of this board, 
we have to replace $a$ and $b$ by $aq^{2n}$ and $bq^n$, respectively, and 
the possible placements of $(m-j)$-rooks in the third part give
$(W_{a,b;q,p}(n))^{m-j}c_{a q^{2n}, b q^n;q,p}(m,j)$.
After placing $(m-j)$ rooks in the upper staircase board part,
there are $j$ many empty columns 
in the rectangular shape part. We place $(j-k+i)$ rooks there.
There are $\binom{j}{j+i-k}=\binom{j}{k-i}$
choices for choosing the columns to place rooks and those rook placements
contribute $([n]_{a,b;q,p})^{j-k+i}$ to the weight sum. This explains all the
terms in \eqref{eqn:conStir1-1}.
A similar argument applies to \eqref{eqn:conStir1-2}
for which we omit the details.
\end{proof}

%%%%%%%%%%%%%%%%%%%%%%%%%%%%%%%%%%%%%%%%%%%%%%%%%%%%%%%%%%%%%%%%%%%%%%%

\subsection{Elliptic extension of the $\alpha$-parameter model}

%%%%%%%%%%%%%%%%%%%%%%%%%%%%%%%%%%%%%%%%%%%%%%%%%%%%%%%%%%%%%%%%%%%%%%%
In \cite{GH}, Goldman and Haglund introduced generalized rook models, called 
\emph{$i$-creation model} and \emph{$\alpha$-parameter model}, which we briefly 
review first. Given a Ferrers board $B$ and a file placement
$P\in \mathcal F_k (B)$, 
we assign weights to the rows containing rooks as follows.
If there are $u$ rooks in a given 
row, then the weight of this row is 
$$
\begin{cases}
1 & \text{ if } 0\le u\le 1,\\
\alpha(2 \alpha-1)(3\alpha -2)\cdots ((u-1)\alpha -(u-2)),&  \text{ if } u\ge 2.
\end{cases}
$$
The weight of a placement $P$, $wt(P)$, is the product of the weights of
all the rows. Then for a Ferrers board $B$, set 
$$r_k ^{(\alpha)}(B)=\sum_{P\in \mathcal F_k (B)}wt(P).$$
Note that for $\alpha =0$, $r_k ^{(0)}(B)$ reduces to the original rook number.
If $\alpha$ is a positive integer $i$, $r_k ^{(i)}(B)$ is the \emph{$i$-creation}
rook number which counts the number of $i$-creation rook placements of
$k$ rooks on $B$. The $i$-creation rook placement is defined as follows:
we first choose the columns to place the rooks. Then as we place rooks
from left to right, each time a rook is placed, $i$ new rows are created
drawn to the right end and immediately above where the rook is placed. 

In this setting Goldman and Haglund \cite{GH} proved the
\emph{$\alpha$-factorization theorem}. 
Given a Ferrers board $B=B(b_1,\dots, b_n)$,
\begin{equation}\label{eqn:alpha}
\prod_{j=1}^n (z+b_j +(j-1)(\alpha -1)) 
= \sum_{k=0}^n r_k ^{(\alpha)}(B)z (z+\alpha -1)\cdots (z+(n-k-1)(\alpha -1)).
\end{equation}
Goldman and Haglund also defined the $q$-analogue of $r_k ^{(\alpha)}(B)$ by 
assigning $q$-weights to the cells in $B$. For a cell $c\in B$, let $v(c)$ be
the number of rooks strictly to the left of, and in the same row as $c$.
Then define the weight of $c$ to be 
$$
wt(c)=
\begin{cases}
 1,& \text{ if there is a rook above and in the same column as $c$,}\\
 [(\alpha -1)v(c)+1]_q, & \text{ if $c$ contains a rook,}\\
 q^{(\alpha -1)v(c) +1},& \text{ otherwise.}
\end{cases}
$$
For a placement $P\in \mathcal F_k (B)$, define the weight of $P$ to be 
$$wt^{(\alpha)}(P) = \prod_{c\in B}wt(c)$$
and 
$$r_k ^{(\alpha)}(q;B)=\sum_{P\in \mathcal F_k (B)}wt^{(\alpha)}(P).$$
Then one has the $q$-analogue of the $\alpha$-factorization theorem 
\begin{equation}\label{eqn:qalpha}
\prod_{j=1}^n [z+b_j +(j-1)(\alpha -1)]_q 
= \sum_{k=0}^n r_{n-k} ^{(\alpha)}(q;B)[z]_q [z+\alpha -1]_q\cdots
[z+(k-1)(\alpha -1)]_q.
\end{equation}

Now, we construct an elliptic extension of \eqref{eqn:qalpha}. Set
$\A_{\alpha,n} =(0, \alpha-1,\dots, \alpha-1)$,
$\B_{\alpha,n} =(b_1, b_2+\alpha-1, b_3 +2(\alpha -1),\dots, b_n+
(n-1)(\alpha -1))$, for $b_1\le b_2\le \cdots \le b_n$, and
$sgn(i)=\overline{sgn}(i)=1$ for all $i=1,\dots, n$.  
In this setting, \eqref{eqn:elptMR} becomes 
\begin{multline}
\prod_{j=1}^n [z+b_j+(j-1)(\alpha -1)]_{aq^{-2(b_j+(j-1)(\alpha-1))},
bq^{-b_j -(j-1)(\alpha-1)};q,p}\\
=\sum_{k=0}^n R_{n-k}^\A (a,b;q,p;
\B_{\alpha,n}^{\A_\alpha,n},sgn,\overline{sgn})
\prod_{i=1}^k[z+(i-1)(\alpha -1)]_{aq^{-2(i-1)(\alpha-1)},bq^{-(i-1)(\alpha-1)};q,p}. 
\end{multline}
The coefficient $R_{k}^\A (a,b;q,p;\B_{\alpha,n}^{\A_\alpha,n},sgn,\overline{sgn})$
can be considered 
as an elliptic extension of $r_k ^{(\alpha)}(q;B)$. 

\begin{remark}
In \cite{GH}, Goldman and Haglund mentioned that in the case
$\alpha=2$ and the staircase board $St_n =B(0,1,2,\dots, n-1)$, 
$$r_k ^{(2)}(q;St_n)=q^{\binom{n-k}{2}}\begin{bmatrix}n+k-1\\2k\end{bmatrix}_q
\prod_{j=1}^k [2j-1]_q.$$
If we set $\A_{2,n} =(0,1,\dots, 1)$, $\B_{2,n}=(0,2,4,\dots, 2(n-1))$ and
$sgn(i)=\overline{sgn}(i)=1$, for all $i=1,\dots, n$, then 
$R_{k}^\A (a,b;q,p;\B_{2,n}^{\A_{2,n}},sgn,\overline{sgn})$ is an
elliptic extension of 
$r_k ^{(2)}(q;St_n)$. Unfortunately
$R_{k}^\A (a,b;q,p;\B_{2,n}^{\A_{2,n}},sgn,\overline{sgn})$ 
does not have a closed form.
However, in the limiting case $p\to 0$ and $b\to 0$ in this order, we have
\begin{multline}\label{eqn:alpha=2}
R_{k}^\A (a,0;q,0;\B_{2,n}^{\A_{2,n}},sgn,\overline{sgn})\\
=q^{-\binom{n+k}{2}+k(k+2)}\begin{bmatrix}n+k-1\\2k\end{bmatrix}_q
\prod_{j=1}^k [2j-1]_q
\frac{(a q^{3-2n+2k};q^2)_{n-k} (a q^{1-2n};q^2)_k}{(a q^{5-4n};q^4)_n}.
\end{multline}
Note that by dividing the cases when there is a rook in the last column
or not, we can derive the following recursion for
$R_{k}^\A (a;q;\B_{2,n}^{\A_{2,n}},sgn,\overline{sgn}):=
R_{k}^\A (a,0;q,0;\B_{2,n}^{\A_{2,n}},sgn,\overline{sgn})$ 
\begin{multline}\label{a-case-rec}
 R_{k}^\A (a;q;\B_{2,n}^{\A_{2,n}},sgn,\overline{sgn})= 
 \frac{W_{a;q}(-n+k+1)}{W_{a;q}(2-2n)}
R_{k}^\A (a;q;\B_{2,n-1}^{\A_{2,n-1}},sgn,\overline{sgn})\\
 +\frac{W_{a;q}(2-2n)[n+k-2]_{aq^{2(2-2n)};q}}{W_{a;q}(2-2n)}
R_{k-1}^\A (a;q;\B_{2,n-1}^{\A_{2,n-1}},sgn,\overline{sgn}).
\end{multline}
Here 
$$W_{a;q}(k)=\frac{(1-a q^{1+2k})}{(1-aq)}q^{-k},\qquad [n]_{a;q}=
\frac{(1-q^n)(1-a q^n)}{(1-q)(1-aq)}q^{1-n}.$$
The coefficients $W_{a;q}(-n+k+1)$ in the first term and $W_{a;q}(2-2n)$
in the denominators in both terms are from the extra factors multiplied
to $M\!R_{k}^{\A}(a;q;\B^{\A},sgn,\overline{sgn})$ in the definition of
$R_{k}^\A (a;q;\B_{2,n}^{\A_{2,n}},sgn,\overline{sgn})$ in \eqref{eqn:Rdef},
and the factor $W_{a;q}(2-2n)[n+k-2]_{aq^{2(2-2n)};q}$ in the second term
comes from placing a rook in the last column of $\B_{2,n}^{\A_{2,n}}$.
The weights assigned to the cells in the base part are 
$$W_{a;q}(-1),W_{a;q}(-2),\dots, W_{a;q}(2-2n)$$
reading from the bottom, and the weights of the cells in the upper
augmented part are 
$$-W_{a;q}(1-n), -W_{a;q}(2-n),\dots, -W_{a;q}(-2),-W_{a;q}(-1)$$
reading from the top. Since the $(k-1)$ rooks placed in the left
$(n-1)$ columns cancel $(k-1)$ cells from the top, the last rook can
be placed in the cells in the base part and the $(n-k)$ cells 
in the upper augmented part, counting from the bottom.
Then the sum of the weights is 
\begin{align*}
 &(W_{a;q}(-1)+W_{a;q}(-2)+\cdots+W_{a;q}(2-2n))\\
 &\quad +(-W_{a;q}(-1)-W_{a;q}(-2)-\cdots - W_{a;q}(-n+k))\\
 =& ~ W_{a;q}(-n+k-1)+\cdots +W_{a;q}(2-2n)\\
 =& ~ W_{a;q}(2-2n)(1+W_{aq^{2(2-2n)};q}(1)+\cdots +W_{aq^{2(2-2n)};q}(n+k-3))\\
 =& ~W_{a;q}(2-2n) [n+k-2]_{aq^{2(2-2n)};q}.
\end{align*}
Note that we used the identity \eqref{eqn:Wid}.
Finally, by appealing to the explicit recursion for
$R_{k}^\A (a;q;\B_{2,n}^{\A_{2,n}},sgn,\overline{sgn})$ in \eqref{a-case-rec},
the formula \eqref{eqn:alpha=2} is readily proved by induction.
\end{remark}

%%%%%%%%%%%%%%%%%%%%%%%%%%%%%%%%%%%%%%%%%%%%%%%%%%%%%%%%%%%%%%%%%%%%%%%5

\bibliographystyle{alpha}
% use the following instead if you encounter problems 
%\bibliographystyle{alpha}
\bibliography{rook}

\begin{thebibliography}{GJW75}

\bibitem[BR06]{BR0}
K.~S. Briggs and J.~B. Remmel.
\newblock {$m$}-rook numbers and a generalization of a formula of {F}robenius
  to {$C_m\wr S_n$}.
\newblock {\em J. Combin. Theory Ser. A}, 113(6):1138--1171, 2006.

\bibitem[BR09]{BR1}
K.~S. Briggs and J.~B. Remmel.
\newblock A {$p,q$}-analogue of the generalized derangement numbers.
\newblock {\em Ann. Comb.}, 13(1):1--25, 2009.

\bibitem[dML95]{ML}
A.~de~M{\'e}dicis and P.~Leroux.
\newblock Generalized {S}tirling numbers, convolution formulae and
  {$p,q$}-analogues.
\newblock {\em Canad. J. Math.}, 47(3):474--499, 1995.

\bibitem[GH00]{GH}
J.~R. Goldman and J.~Haglund.
\newblock Generalized rook polynomials.
\newblock {\em J. Combin. Theory Ser. A}, 91(1-2):509--530, 2000.
\newblock In memory of Gian-Carlo Rota.

\bibitem[GJW75]{GJW}
J.~R. Goldman, J.~T. Joichi, and D.~E. White.
\newblock Rook theory. {I}. {R}ook equivalence of {F}errers boards.
\newblock {\em Proc. Amer. Math. Soc.}, 52:485--492, 1975.

\bibitem[GR]{GR1}
A.~M. Garsia and J.~B. Remmel.
\newblock unpublished work.

\bibitem[GR86]{GR}
A.~M. Garsia and J.~B. Remmel.
\newblock {$Q$}-counting rook configurations and a formula of {F}robenius.
\newblock {\em J. Combin. Theory Ser. A}, 41(2):246--275, 1986.

\bibitem[GR04]{GRhyp}
G.~Gasper and M.~Rahman.
\newblock {\em Basic hypergeometric series}, volume~96 of {\em Encyclopedia of
  Mathematics and its Applications}.
\newblock Cambridge University Press, Cambridge, second edition, 2004.
\newblock With a foreword by Richard Askey.

\bibitem[HR01]{HR}
J.~Haglund and J.~B. Remmel.
\newblock Rook theory for perfect matchings.
\newblock {\em Adv. in Appl. Math.}, 27(2-3):438--481, 2001.
\newblock Special issue in honor of Dominique Foata's 65th birthday
  (Philadelphia, PA, 2000).

\bibitem[MR08]{MR}
B.~K. Miceli and J.~B. Remmel.
\newblock Augmented rook boards and general product formulas.
\newblock {\em Electron. J. Combin.}, 15(1):Research Paper 85, 55, 2008.

\bibitem[RW04]{RW}
J.~B. Remmel and M.~L. Wachs.
\newblock Rook theory, generalized {S}tirling numbers and {$(p,q)$}-analogues.
\newblock {\em Electron. J. Combin.}, 11(1):Research Paper 84, 48, 2004.

\bibitem[Sch]{Schl1}
M.~J. Schlosser.
\newblock A noncommutative weight-dependent generalization of the binomial
  theorem.
\newblock {\em preprint}, arXiv:1106.2112.

\bibitem[SY]{SY}
M.~J. Schlosser and M.~Yoo.
\newblock Elliptic rook and file numbers.
\newblock {\em preprint}.

\end{thebibliography}
\label{sec:biblio}

%%%%%%%%%%%%%%%%%%%%%%%%%%%%%%%%%%%%%%%%%%%%%%%%%%%%%%%%%%%%%%%%%%%%%%%5

\end{document}